\documentclass[11pt,leqno]{article}
\usepackage{amsmath,amsthm, amsfonts}
\usepackage{amsfonts}
\usepackage{mathrsfs}
\usepackage{hyperref}
\usepackage{amssymb}
\usepackage{epsfig}
\usepackage{subfig}
\usepackage{graphicx}
\usepackage{bm}
\usepackage{float}
\usepackage{booktabs}
\numberwithin{equation}{section} 
\newtheorem{theorem}{\bf Theorem}[section]
\newtheorem{example}{\bf Example}[section]

\newtheorem{remark}{\bf Remark}[section]
\newtheorem{lemma}{\bf Lemma}[section]

\newcommand{\norm}[1]{\left\lVert #1\right\rVert}

\newsavebox{\savepar}

\begin{document}
\title{\bf \Large Global stabilization and finite element analysis of the viscous Burgers' equation with memory subject to  Neumann boundary feedback control}
\author{
	Shishu Pal Singh\thanks{Department of Mathematical Sciences, Rajiv Gandhi Institute of Petroleum Technology. Email: shishups@rgipt.ac.in}
	\;and
		Sudeep Kundu\thanks{Department of Mathematical Sciences, Rajiv Gandhi Institute of Petroleum Technology. Email: sudeep.kundu@rgipt.ac.in}  
}
\date{\today}
 \maketitle
 \begin{abstract}
 This paper presents a global stabilization result of the viscous Burgers' equation with the memory term by applying Neumann boundary feedback control laws. We construct suitable feedback control inputs using the control Lyapunov functional and establish stabilization in the  \(L^{2}, H^{1},\) and \(H^{2}\)-norms. The existence and uniqueness of the solution are established through the  Faedo-Galerkin method. Moreover, we show the global stabilization where the diffusion coefficient $\nu$ is unknown. Then, we apply a \(C^{0}\)-conforming finite element method to the spatial variable while keeping the time variable continuous. Furthermore, we obtain global stabilization of the semi-discrete scheme and optimal error estimates for the state variable in the \(L^{\infty}\), \(L^{2}\), and  \(H^{1}\)-norms, using the Ritz-Volterra projection. Additionally, error estimates for the feedback control laws are established. Lastly, we present some numerical simulations to demonstrate the theoretical findings.
 \end{abstract}

 \textbf{MSC Classification (2020):} 93D15, 93B52, 35B37, 65M60, 65M15. 
\section{Introduction.}
 In this article, we are concerned with the Neumann boundary feedback control inputs of the viscous Burgers' equation with memory
\begin{align}
	y_{t} (x, t) + y (x, t) y_x (x, t) - \nu y_{{xx}} (x, t) - \rho
	\int_0^t e^{- \delta (t - s)} y_{xx} (x, s) {ds} &= 0 \quad (x,
	t) \in (0, 1) \times (0, \infty),  &  &  \label{eq:1.1}\\
	y_x (0, t) = u_{0} (t) \quad t \in (0, \infty),  &  &  \label{eq:1.2}\\
	y_x (1, t) = u_{1}(t) \quad t \in (0, \infty),&  &  \label{eq:1.3}\\
	y(x,0) = y_{0}(x) \quad x\in(0,1).\label{eq:1.4}
\end{align}
Here, $\nu$, $\rho,$ and $ \delta$  are positive constants; $u_{0}(t)$ and  $u_{1}(t)$ are boundary feedback control inputs and $y_{0}(x)$ is a given initial datum. Through a linear time convolution involving the diffusion term $y_{xx}(.,.)$ and the memory kernel $e^{-\delta t}$, the influence of the memory term is incorporated.


For $\rho=0$, this model \eqref{eq:1.1}-\eqref{eq:1.4} is known as the viscous Burgers' equation with Neumann boundary,
 which has various physical phenomena such as gas dynamics, turbulent flows, boundary layer problems and traffic flow problems \cite{Abbasi2018} etc. In the case of nonzero $\rho$,  the state variable depends on the past history of the system. Additionally, this equation with memory allows for modeling of the viscoelastic behavior\cite{MR4408112}, the study of heat conduction \cite{Gurtin1968}, bidirectional nonlinear shallow water waves \cite{MR828198}, nuclear reactor dynamics, and population dynamics \cite{ZHAO2023575} etc. The parabolic equation with memory is an important problem in control theory that attracts the attention of many researchers.

 In control theory, the work related to stabilization of the viscous Burgers' type equation has been going on for the last few decades. Local stabilization results  are found in \cite{MR1662924, Burns1991, MR1269994} and \cite{MR1652947}. In \cite{MR1442024}, the authors examine the existence and uniqueness of the viscous Burgers' equation for any initial data in the \(L^{2}\)-norm. The global stabilization of the viscous Burgers' equation with Neumann and Dirichlet boundary control is established by Krstic \cite{MR1751258} using the control Lyapunov functional. After that, the authors discuss stabilization in the \(H^{1}\)-norm \cite{balogh1999global} and well-posedness of the classical solution. 
In \cite{liu2001adaptive, MR2091456, MR2146486}, stabilization results of the viscous Burgers' equation for both adaptive (where \(\nu\) is unknown) and non-adaptive (when \(\nu\) is known) cases are proved using distributed and boundary feedback control inputs.
 
 The recent paper \cite{MR4408112} provides local stabilization of the viscous Burgers' equation with memory with distributed feedback control input. The stabilization of the heat equation with a memory, subject to the Dirichlet boundary feedback control is established in \cite{MR3818113, MR3338307}. In the case of distributed feedback control, Sistla et al. \cite{sistla2025lqr} discuss the so called \(\omega\)-stabilization of the heat equation with a memory. Furthermore, stabilization of the class of linear parabolic equation with memory is analyzed using projection based feedback input operators in \cite{khan2025dynamic} and for
 related to parabolic integro-differential (PID) equation has been shown in \cite{MR1045774}. In \cite{bag2025memory}, the authors prove stabilization of the generalized Burgers'-Huxley equation with memory by the Dirichlet boundary control input.
  
  Using a \(C^{0}\)-conforming finite element method, the error analysis for the state variable in the \(L^{\infty}(L^{\infty}), L^{\infty}(L^{2}),\) and \(L^{\infty}(H^{1})\)-norms of the semi-discrete scheme for the viscous Burgers' equation is established in \cite{ MR3790146} keeping the time variable continuous. Super-convergence results for the feedback control laws are also proved. Moreover, the results related to the PID equation without control have been well established. For example, in \cite{MR1225703}, using a finite element method in the spacial direction and the backward Euler method in the temporal direction, error analysis with a weakly singular kernel is proved. The optimal error estimates of a semi-discrete scheme are analyzed in \cite{MR1631540} for the general PID equation using a Galerkin finite element method. Using an $H^{1}$-Galerkin Mixed finite element method, the error of numerical solution to the evolution equation with a positive kernel is discussed in \cite{S0036142900372318}. Moreover, the PID equation using a finite difference method and the cubic spline collocation method has been solved in \cite{NETA1985607} and \cite{fairweather1994spline}, respectively.
 
 In this article, we pursue two main objectives. First, global stabilizations for the viscous Burgers' equation with memory term under the Neumann boundary feedback control laws are established. Second, we derive error estimates for the state variable and feedback control laws, using a finite element method.
 Here, we find boundary feedback control laws that achieve global asymptotic stability for the viscous Burgers' equation with memory. Moreover, applying a \(C^{0}\)-conforming finite element method to this equation and keeping the time variable continuous, we discuss the error analysis for the state variable and the control inputs. 
 
 The major contribution of this article are as follows:
 \begin{itemize}
 	\item First, we derive boundary feedback control inputs using the control Lyapunov functional. With the help of the Faedo-Galerkin method, we show the existence and uniqueness of the feedback control problem \eqref{eq:4.1}-\eqref{eq:4.8}. By applying boundary feedback control inputs, we discuss the global stabilization in the \(L^{2}, H^{1},\) and \(H^{2}\)-norms and exponential bounds for the state variable. Moreover, we establish some regularity estimates of the state variable.
 	\item When the diffusion coefficient \(\nu\) is unknown, we analyze the adaptive feedback control inputs and find identifiers at both end points using the Lyapunov approach.
 	\item A \(C^{0}\)-conforming finite element method is applied to the space variable keeping the time variable continuous, and the global stabilization of the semi-discrete scheme is shown. We establish the corresponding error analysis for the state variable and control inputs, using the Ritz-Volterra projection. In addition, the error for the state variable is analyzed in the \(L^{\infty}, L^{2},\) and \(H^{1}\)-norms.
 	\item Finally, we present some numerical examples to demonstrate the effects of the memory term's and other parameters.
 \end{itemize}

In this paper, we use the following notations:

 The space $ L^{p}((0,T);X)$ consists of all strongly measurable functions $ f:[0,T]\rightarrow X $, equipped with the norm
\begin{align*}
	\norm{f}_{L^{p}((0,T);X)}:=
	\begin{cases}
		\left(\int_{0}^{T}\norm{f}_{X}^{p}dt\right)^{\frac{1}{p}}< \infty, & \ 1 \leq p<\infty, \\
		\mathop{\mathrm{ess\, sup}} \limits_{0\leq t\leq T} (\norm{f(t)}_{X})< \infty, \medspace & \  p=\infty,
	\end{cases}
\end{align*}
where $ X $ denotes a Banach space with the norm $ {||.||}_{X}$. Denote  $ {( .,. )}$  the $ L^{2}$-inner product with the corresponding norm denoted by \( \norm{\cdot} \).  For convenience, we write \( L^{p}((0,T);X) \) as \( L^{p}(X) \). Moreover, denote \(H^{m}(0,1)\) as the standard Sobolev space with the norm \(\norm{\cdot}_{m}\).
\begin{itemize}
	\item  \textbf{Gr\"onwall's inequality\cite{ MR3790146}:} Let \(A(t)\) be a non-negative, absolutely continuous function on \([0, \infty)\) that satisfies for a.e. \(t\) the differential inequality 
	\[ \frac{d}{dt}A(t)+B(t)\leq C(t)+K(t)A(t),
	\]
	where \(B, C,\) and \(K\) are non-negative, locally integrable function on \([0, \infty)\). Then 
	\[0\leq A(t)+\int_{0}^{t}B(s)ds\leq \Big(A(0)+\int_{0}^{t}C(s)ds\Big)\exp\big(\int_{0}^{t}K(s)ds\big), \forall t>0.
	\]
	\item \textbf{Poincar\'e-Wirtinger's Inequality\cite{ MR3790146}:} For \(f\in H^{1}(0,1)\), there holds
	\[ \norm{f}^{2}\leq 2\norm{ f_{x}}^{2}+ f^{2}(i), \quad i=0,1,
	\]
	and \[\norm{f}_{\infty}\leq \sqrt{2}\norm{|f|},
	\]
	where \(\norm{|f|}:=\sqrt{\norm{f_{x}}^{2}+f^{2}(0)+f^{2}(1)} \), which is equivalent to the $H^{1}$-norm.
	\item \textbf{Young's Inequality:} For all \( a, b > 0 \) and \( \epsilon > 0 \), we have  
	\[
	ab \leq \frac{\epsilon}{2} a^{2} + \frac{1}{2\epsilon} b^{2}.
	\]
\end{itemize}

The steady state problem of the closed loop system \eqref{eq:1.1}-\eqref{eq:1.4} is to find $ y^{\infty} $ such that 
\begin{align}
	-\left(\nu+\frac{\rho}{\delta}\right) y_{xx}^{\infty} + y^{\infty}y_{x}^{\infty} =0,\label{eq:3.1}\\
	y_{x}^{\infty}(0) = y_{x}^{\infty}(1) = 0.\label{eq:3.2}
\end{align}
 The steady state problem \eqref{eq:3.1}- \eqref{eq:3.2} is equivalent to the steady state of the viscous Burgers' equation with Neumann boundary. Then, the steady state solution of the problem \eqref{eq:3.1}- \eqref{eq:3.2} is any constant (say $ w_{d}$). For more details see \cite{ALLEN20021165, MR3100771, Burns1991}.

Without loss of generality, we assume that $w_{d}\geq 0$. Our goal is 
\[ \lim_{t\rightarrow\infty} y(x,t) = w_d\quad  \forall\,\, x \in [0,1].\]

Hence, it is enough to stabilize around $w=0$, where we consider $w = y - w_d$ such that $w=0$ as $t\rightarrow \infty$. Introducing a new variable to write the system \eqref{eq:1.1}-\eqref{eq:1.4} in the couple form for finding feedback control laws
\begin{align}
\label{1.8}
    z(x,t) := \int_{0}^{t}e^{-\delta(t-s)}w_{x}(x,s)ds.
\end{align}
From \eqref{eq:1.1}-\eqref{eq:1.4} and \eqref{eq:3.1}-\eqref{eq:3.2}, we have the system in the following form after differentiating \eqref{1.8} with respect to time
\begin{align}
	w_t(x,t) - \nu w_{xx}(x,t) + w(x,t)w_x(x,t) + w_d w_x(x,t) &= \rho z_x(x,t) \quad (x,t) \in (0,1)\times(0, \infty),\label{eq:4.1}\\
	z_t(x,t) + \delta z(x,t) &=  w_x(x,t) \quad (x,t) \in (0, 1) \times (0, \infty),\label{eq:4.2}\\
	w_x(0,t) &= v_0(t) \quad t\in (0, \infty),\label{eq:4.3}\\
	w_x(1,t) &= v_1(t) \quad t\in (0, \infty),\label{eq:4.4}\\
	w(x,0)& =y_{0}(x)-w_{d}=: w_0(x) (\text{say}),\quad x\in(0,1),\label{eq:4.5}\\
	z(x,0) &= 0, \quad x\in[0,1].\label{eq:4.8}
\end{align}
To find the feedback control laws, we consider the following control Lyapunov functional
\begin{align*}
	V &:= \frac{1}{2}\int_{0}^{1}\big(w^2(x,t) +\rho z^2(x,t)\big) dx.
\end{align*}
Differentiating with respect to time, we obtain
\begin{align*}
	\frac{d{V}}{dt} &= \int_{0}^{1}w(x,t)w_t(x,t)dx+ \rho \int_{0}^{1}z(x,t)z_t(x,t)dx.
\end{align*}
Therefore, from \eqref{eq:4.1} and \eqref{eq:4.2}, it follows that 
\begin{align*}
	\frac{dV}{dt}= \int_{0}^{1}w(x,t)\Big(\rho z_x(x,t) &+ \nu w_{xx} (x,t)- w(x,t)w _x(x,t) - w_dw_x(x,t)\Big)dx 
	 \\& + \rho \int_{0}^{1}z(x,t)( w_x(x,t) - \delta z(x,t))dx.
\end{align*}
Using integration by parts, we see that
\begin{align}
	\label{eq:5}
	\frac{dV}{dt} = \rho w(1,t)z(1,t) &- \rho w(0,t)z(0,t) - \rho \int_{0}^{1}w_x(x,t) z(x,t) dx +\nu w(1,t)w_x(1,t)-\nu w(0,t)w_x(0,t) \nonumber\\&- \nu \int_0^1w_x^2(x,t)dx -\frac{1}{3}\Big(w^3(1,t)-w^3(0,t) \Big)
	\nonumber\\&-\frac{w_d}{2}\Big(w^2(1,t)-w^2(0,t)\Big)+ \rho \int_0^1z(x,t)w_x(x,t)dx -\rho\delta \int_0^1z^2(x,t)dx.
\end{align}
A use of the Cauchy-Schwarz and Young's inequality yields
	\begin{align*}
		\frac{1}{3}w^{3}(0,t)\leq \frac{c_{0}}{2}w^{2}(0,t)+\frac{1}{18c_{0}}w^{4}(0,t),\\
		-\frac{1}{3}w^{3}(1,t)\leq \frac{c_{1}}{2}w^{2}(1,t)+\frac{1}{18c_{1}}w^{4}(1,t).
	\end{align*}
	Using these inequality in \eqref{eq:5}, we obtain
\begin{align}
	\frac{dV}{dt}&\leq  w(1,t)\bigg( \nu w_x(1,t) +(c_{1} + w_{d})w(1,t)+ \frac{2}{9c_1}w^3(1,t)+\rho z(1,t)\bigg) -\frac{3w_d}{2}w^2(1,t) - \frac{c_{1}}{2} w^{2}(1,t)
	\nonumber\\
	&\quad -\frac{1}{6c_1}w^4(1,t) - w(0,t)\bigg( \nu w_x(0,t) -( c_{0} + w_{d})w(0,t)- \frac{2}{9c_{0}}w^3(0,t)+\rho z(0,t) \bigg) - \frac{1}{6c_0}w^{4}(0,t)\nonumber\\
	& \quad -\frac{w_d}{2}w^2(0,t)-\frac{c_{0}}{2} w^2(0,t) -\nu\norm{w_{x}}^2 -\rho\delta \norm{z}^2. \nonumber
\end{align}
We choose feedback control laws in the following form
\begin{align}
	v_0(t) = \frac{1}{\nu}\bigg((c_0 + w_d)w(0,t) + \frac{2}{9c_0}w^3(0,t) - \rho z(0,t)\bigg),\label{v0}\\
	v_1(t) = -\frac{1}{\nu}\bigg((c_1 + w_d)w(1,t) + \frac{2}{9c_1}w^3(1,t) + \rho z(1,t)\bigg),\label{v1}
\end{align}
therefore, we arrive at 
\[ \frac{dV}{dt} \leq -\nu\norm{w_{x}}^2 -\rho\delta \norm{z}^2-\frac{w_d}{2}w^2(0,t)-\frac{c_{0}}{2} w^2(0,t)-\frac{3w_d}{2}w^2(1,t) - \frac{c_{1}}{2} w^{2}(1,t).
\]
Consequently, we can write
\[\frac{dV}{dt} \leq -\min\Big(\nu, \delta,\frac{(w_{d}+c_{0})}{2},\frac{(3w_{d}+c_{1})}{2}\Big)\Big(\norm{w}^{2}+\rho\norm{z}^{2}\Big)=-2\min\Big(\nu, \delta,\frac{(w_{d}+c_{0})}{2},\frac{(3w_{d}+c_{1})}{2}\Big)V,
\]
where $\nu, w_{d}, c_{0}, \delta,$ and $c_{1}$ are positive constants.

The weak formulation of the system \eqref{eq:4.1}-\eqref{eq:4.8} is to find $ w\in H^{1}(0,1)$  such that 
\begin{align}
	\label{1.13}
	\nonumber(w_{t},\phi) +\nu (w_{x},\phi_{x})+w_{d}(w_{x},\phi) +(ww_{x},\phi)+ \left((c_{0}+w_{d})w(0,t)+  \frac{2}{9c_{0}}w^{3}(0,t)-\rho z(0,t)\right)\phi(0)\\ +\left((c_{1}+w_{d})w(1,t)+ \frac{2}{9c_{1}}w^{3}(1,t)+\rho z(1,t)\right)\phi(1)=\rho (z_{x}, \phi), \hspace{5mm} \forall \  \phi \in H^{1},
\end{align}
and 
\begin{align}
	\label{1.14}
	z_{t}+\delta z= w_{x},
\end{align}
with $ w(x,0)=w_{0}(x)$ and $ z(x,0)=0.$

We write the system \eqref{1.13}-\eqref{1.14} in the equivalent form
\begin{align}
	\label{6.1}
	(w_{t},\phi) +\nu (w_{x},\phi_{x})+w_{d}(w_{x},\phi) +(ww_{x},\phi)+ \left((c_{0}+w_{d})w(0,t)+  \frac{2}{9c_{0}}w^{3}(0,t)\right)\phi(0)\\ \nonumber+\left((c_{1}+w_{d})w(1,t)+ \frac{2}{9c_{1}}w^{3}(1,t)\right)\phi(1)+\rho (z, \phi_{x})=0, \hspace{5mm} \forall \  \phi \in H^{1},
\end{align}
where \(z\) is given in \eqref{1.8}.

Throughout the paper, we assume that \(C\) is a generic positive constant. 

We consider the following assumption, which is needed in the proof of stabilization for the state variable.\\
 \textbf{(A1.)}
Compatibility conditions at \(t=0\), \(w_{0x}(0)=v_{0}(0), w_{0x}(1)=v_{1}(0), w_{0xt}(0)=v_{0t}(0)\), and \(w_{0xt}(1)=v_{1t}(0)\) with \(w_{0}\in H^{3}(0,1).\)\\
Further, we assume the following regularity result, which is proved in Section \ref{3}.\\
\textbf{(A2.)}
Let \(w_{0}\in H^{3}(0,1).\) Then there exists a unique weak solution \(w\) such that 
\begin{align*}
	\norm{w}_{2}^{2}+\norm{w_{t}}_{1}^{2}+\norm{z}^{2}+\int_{0}^{t}\Big(\norm{|w(s)|}^{2}+\norm{w_{t}(s)}_{2}^{2}+\norm{z(s)}^{2}\Big)ds\leq C\norm{w_{0}}_{3}^{2}.
\end{align*}

The rest of the paper is organized as follows. Section \ref{2} contains the stabilization of the problem \eqref{eq:4.1}-\eqref{eq:4.8} and some regularity results. Moreover, it presents the stabilization of the state variable in the \(H^{2}\)-norms. Furthermore, we discuss adaptive feedback control when the diffusion coefficient is unknown. Section \ref{3} shows the existence and uniqueness of the problem \eqref{eq:4.1}-\eqref{eq:4.8}. In Section \ref{4}, we discuss a \(C^{0}\)-conforming finite element method and show global stabilization of the semi-discrete scheme. Additionally, error analysis of the semi-discrete scheme for the state variable and control inputs is established.  Section \ref{5} is devoted to some numerical examples. Finally, in Section \ref{6} concluding remarks are given.

\section{Stabilization.}
\label{2}
In this section, we deal with the stabilization of the closed loop system \eqref{eq:4.1}-\eqref{eq:4.8} applying feedback control laws \eqref{v0} and \eqref{v1} in the \(L^{2}, H^{1},\) and \(H^{2}\)-norms. Moreover, we establish some regularity results, which are necessary for the proof of error analysis of the state variable and control inputs.
\begin{lemma}
	\label{L2.1}
	Let the assumptions \((A1)\) and \((A2)\) hold. Then, there exists \( 0\leq\alpha \leq \frac{1}{2}\min\{\nu, \delta,  c_{0}+w_{d}, c_{1}+w_{d}\}\) such that
	\begin{align*}
		\norm{e^{\alpha t}w}^{2}&+\rho\norm{e^{\alpha t}z}^{2}+\beta \int_{0}^{t}e^{2\alpha s}\left(\norm{w_{x}(s)}^{2}+\rho\norm{z(s)}^{2}+w^{2}(0,s)+w^{2}(1,s) \right)ds \nonumber\\&+ \frac{1}{3c_{0}}\int_{0}^{t}e^{2\alpha s}w^{4}(0,s)ds +\frac{1}{3c_{1}}\int_{0}^{t}e^{2\alpha s}w^{4}(1,s)ds\leq \norm{w_{0}}^{2},
	\end{align*}
	where $ 0<\beta= \min\{(\nu-2\alpha), (\delta-2\alpha), (c_{0}+w_{d}-2\alpha), (c_{1}+w_{d}-2\alpha)   \}.$  
\end{lemma}
\begin{proof}
	Setting $ \phi=w $ in \eqref{1.13} and forming the $ L^{2}$-inner product between  \eqref{1.14} and $z$ gives
	\begin{align}
		\label{1.15}
		\frac{1}{2}\frac{d}{dt}\norm{w}^{2}&+ \nu\norm{w_{x}}^{2}+\frac{w_{d}}{2}\left(w^{2}(1,t)-w^{2}(0,t)\right)+\frac{1}{3}\left(w^{3}(1,t)-w^{3}(0,t)\right)\nonumber\\&+(c_{0}+w_{d})w^{2}(0,t)+  \frac{2}{9c_{0}}w^{4}(0,t)-\rho z(0,t)w(0,t)+ \rho z(1,t)w(1,t)\nonumber\\&+(c_{1}+w_{d})w^{2}(1,t)+ \frac{2}{9c_{1}}w^{4}(1,t)=\rho(z_{x}, w),
\end{align}
	and 
	\begin{align}
		\label{1.16}
		\frac{1}{2}\frac{d}{dt}\norm{z}^{2}+\delta\norm{z}^{2}= (w_{x}, z).
	\end{align}
	Using integration by parts on the right hand side of \eqref{1.16}, we arrive at
	\begin{align}
		\label{1.17}
		\frac{1}{2}\frac{d}{dt}\norm{z}^{2}+\delta\norm{z}^{2}= \left(z(1,t)w(1)-z(0,t)w(0)\right)- (z_{x}, w).
	\end{align}
	Substituting the value of $  (z_{x}, w) $ into \eqref{1.15}, we observe that
	\begin{align}
		\label{1.18}
		\frac{1}{2}\frac{d}{dt}\left(\norm{w}^{2}+\rho\norm{z}^{2}\right) &+ \nu\norm{w_{x}}^{2}+\rho\delta\norm{z}^{2} +\frac{w_{d}}{2}\left(w^{2}(1,t)-w^{2}(0,t)\right)\nonumber\\&+(c_{0}+w_{d})w^{2}(0,t)+  \frac{2}{9c_{0}}w^{4}(0,t)+(c_{1}+w_{d})w^{2}(1,t)+ \frac{2}{9c_{1}}w^{4}(1,t)\nonumber\\&=-\frac{1}{3}\left(w^{3}(1,t)-w^{3}(0,t)\right).
	\end{align}
	A use of the Cauchy-Schwarz inequality and Young's inequality on the right hand side of \eqref{1.18} yields
	\begin{align*}
		\frac{1}{3}w^{3}(0,t)\leq \frac{c_{0}}{2}w^{2}(0,t)+\frac{1}{18c_{0}}w^{4}(0,t),\\
		\frac{1}{3}w^{3}(1,t)\leq \frac{c_{1}}{2}w^{2}(1,t)+\frac{1}{18c_{1}}w^{4}(1,t).
	\end{align*}
	Substituting these values into \eqref{1.18}, we obtain
	\begin{align*}
		\frac{1}{2}\frac{d}{dt}\left(\norm{w}^{2}+\rho\norm{z}^{2}\right) &+ \nu\norm{w_{x}}^{2}+\rho\delta\norm{z}^{2}+ \frac{1}{2}\left( (c_{0}+w_{d})w^{2}(0,t)+  \frac{1}{3c_{0}}w^{4}(0,t) \right)\nonumber\\&+\frac{1}{2}\left( (c_{1}+3w_{d})w^{2}(1,t)+ \frac{1}{3c_{1}}w^{4}(1,t)\right)\leq 0.
	\end{align*}
	Multiplying by $ 2e^{2\alpha t} $ to the above equation, it follows that
	\begin{align}
		\label{1.19}
		\frac{d}{dt}\left(\norm{e^{\alpha t}w}^{2}+\rho\norm{e^{\alpha t}z}^{2}\right)&-2\alpha e^{2\alpha t}\left(\norm{w}^{2}+\rho \norm{z}^{2}\right)+2e^{2\alpha t} \nu\norm{w_{x}}^{2}\nonumber\\&+2\rho\delta e^{2\alpha t}\norm{z}^{2}+e^{2\alpha t}\left( (c_{0}+w_{d})w^{2}(0,t)+  \frac{1}{3c_{0}}w^{4}(0,t) \right)\nonumber\\&+e^{2\alpha t}\left( (c_{1}+3w_{d})w^{2}(1,t)+ \frac{1}{3c_{1}}w^{4}(1,t)\right)\leq 0.
	\end{align}
	Using Poincar\'e-Wirtinger's inequality in \eqref{1.19}, we arrive at
	\begin{align*}
		\frac{d}{dt}\left(\norm{e^{\alpha t}w}^{2}+\rho\norm{e^{\alpha t}z}^{2}\right)&+2(\nu-\alpha)e^{2\alpha t}\norm{w_{x}}^{2}+2\rho(\delta-\alpha)e^{2\alpha t}\norm{z}^{2}\nonumber\\&+e^{2\alpha t}\left( (c_{0}+w_{d}-2\alpha)w^{2}(0,t)+  \frac{1}{3c_{0}}w^{4}(0,t) \right)\nonumber\\&+e^{2\alpha t}\left( (c_{1}+3w_{d}-2\alpha)w^{2}(1,t)+ \frac{1}{3c_{1}}w^{4}(1,t)\right)\leq 0.
	\end{align*}
	Integrating with respect to time $ t $ from $ 0 $ to $ t $ gives
	\begin{align*}
		\norm{e^{\alpha t}w}^{2}&+\rho\norm{e^{\alpha t}z}^{2}+\beta \int_{0}^{t}e^{2\alpha s}\left(\norm{w_{x}(s)}^{2}+\rho\norm{z(s)}^{2}+w^{2}(0,s)+w^{2}(1,s) \right)ds \nonumber\\&+ \frac{1}{3c_{0}}\int_{0}^{t}e^{2\alpha s}w^{4}(0,s)ds +\frac{1}{3c_{1}}\int_{0}^{t}e^{2\alpha s}w^{4}(1,s)ds\leq \norm{w_{0}}^{2},
	\end{align*}
	where $ 0<\beta= \min\{(\nu-2\alpha), (\delta-2\alpha), (c_{0}+w_{d}-2\alpha), (c_{1}+w_{d}-2\alpha)   \},$ \\ and $0<\alpha \leq \frac{1}{2}\min\{\nu, \delta,  c_{0}+w_{d}, c_{1}+w_{d}\}.$
\end{proof}
The next lemma discusses the stabilization in the \(H^{1}\)-norm for the state variable.
\begin{lemma}
	\label{L2.2}
	 Under the assumptions \((A1)\) and \((A2)\), it holds
	\begin{align*}
		\frac{\nu }{2}\norm{w_{x}}^{2}+E(t)+e^{-2\alpha t}\int_{0}^{t}e^{2\alpha s}\norm{w_{t}(s)}^{2}ds\leq Ce^{-2\alpha t}\norm{w_{0}}_{1}^{2} e^{C\norm{w_{0}}^{2}},
	\end{align*}
where \[E(t)=\sum_{i=0}^{1}\Big((c_{i}+w_{d})w^{2}(i,t)+\frac{1}{9c_{i}}w^{4}(i,t)\Big).
\]
\end{lemma}
\begin{proof}
	Select \(\phi=w_{t}\) in \eqref{1.13} to obtain
	\begin{align*}
		\frac{\nu}{2}\frac{d}{dt}\norm{w_{x}}^{2}&+\norm{w_{t}}^{2}+ \sum_{i=0}^{1}\Big((c_{i}+w_{d})w(i,t)w_{t}(i,t)+\frac{2}{9c_{i}}w^{3}(i,t)w_{t}(i,t)\Big)\\&-\rho z(0,t)w_{t}(0,t)+\rho z(1,t)w_{t}(1,t)=-w_{d}(w_{x},w_{t})-(w w_{x}, w_{t})+\rho (z_{x}, w_{t}).
	\end{align*}
Using integration by parts in the last term on the right hand side  yields
\begin{align}
	\label{2.11}
	\frac{\nu}{2}\frac{d}{dt}\norm{w_{x}}^{2}+\norm{w_{t}}^{2}+ \sum_{i=0}^{1}\Big((c_{i}+w_{d})w(i,t)w_{t}(i,t)&+\frac{2}{9c_{i}}w^{3}(i,t)w_{t}(i,t)\Big)\nonumber\\&=-w_{d}(w_{x},w_{t})-(w w_{x}, w_{t})-\rho (z, w_{xt}).
\end{align}
Since 
\begin{align*}
 - (z, w_{xt})=-\delta(z,w_{x})-\frac{d}{dt}(z,w_{x})+\norm{w_{x}}^{2},
\end{align*}
 from \eqref{2.11}, we arrive at 
\begin{align*}
	\frac{\nu}{2}\frac{d}{dt}\norm{w_{x}}^{2}&+\norm{w_{t}}^{2}+ \sum_{i=0}^{1}\Big((c_{i}+w_{d})w(i,t)w_{t}(i,t)+\frac{2}{9c_{i}}w^{3}(i,t)w_{t}(i,t)\Big)\nonumber\\&=-w_{d}(w_{x},w_{t})-(w w_{x}, w_{t})-\rho \delta(z,w_{x})-\rho \frac{d}{dt}(z,w_{x})+\rho \norm{w_{x}}^{2}.
\end{align*}
Using the Cauchy-Schwarz inequality and Young's inequality, we deduce that
\begin{align*}
	\frac{1}{2}\frac{d}{dt}\Big(\nu \norm{w_{x}}^{2}+E(t)\Big)+\frac{1}{2}\norm{w_{t}}^{2}\leq C\norm{w_{x}}^{2}+ C\norm{z}^{2}+ C\norm{w}_{\infty}^{2}\norm{w_{x}}^{2}-\rho \frac{d}{dt}(z,w_{x}),
\end{align*}
where \(E(t)=\sum \limits_{i=0}^{1}\Big((c_{i}+w_{d})w^{2}(i,t)+\frac{1}{9c_{i}}w^{4}(i,t)\Big)\).

Multiplying by \(2e^{2\alpha t}\) to the above inequality gives
\begin{align*}
	\frac{d}{dt}\Big(\nu \norm{e^{\alpha t}w_{x}}^{2}+e^{2\alpha t}E(t)\Big)+e^{2\alpha t}\norm{w_{t}}^{2}&\leq Ce^{2\alpha t}\norm{w_{x}}^{2}+ Ce^{2\alpha t}\norm{z}^{2}+ Ce^{2\alpha t}\norm{|w|}^{2}\norm{w_{x}}^{2}
	\\&\quad -\rho \frac{d}{dt}\big(2e^{2\alpha t}(z,w_{x})\big)+2\alpha \rho e^{2\alpha t}(z,w_{x})\\&\quad+2\alpha e^{2\alpha t}\Big(\nu \norm{w_{x}}^{2}+E(t)\Big).
\end{align*}
Integrating with respect to time over \([0,t]\) and using Lemma \ref{L2.1}, we get
\begin{align*}
	\nu \norm{e^{\alpha t}w_{x}}^{2}+e^{2\alpha t}E(t)+\int_{0}^{t}e^{2\alpha s}\norm{w_{t}(s)}^{2}ds&\leq C\norm{w_{0}}_{1}^{2}+ C\int_{0}^{t}e^{2\alpha s}\norm{|w(s)|}^{2}\norm{w_{x}(s)}^{2}ds
	\\&\quad -\rho\big(2e^{2\alpha t}(z,w_{x}).
\end{align*}
Again, using the Cauchy-Schwarz inequality and Young's inequality, it follows that
\begin{align*}
	\frac{\nu }{2}\norm{e^{\alpha t}w_{x}}^{2}+e^{2\alpha t}E(t)+\int_{0}^{t}e^{2\alpha s}\norm{w_{t}(s)}^{2}ds&\leq  C\int_{0}^{t}e^{2\alpha s}\norm{|w(s)|}^{2}\Big(\norm{w_{x}(s)}^{2}+E(s)\Big)ds\\&\quad +Ce^{2\alpha t}\norm{z}^{2}+C\norm{w_{0}}_{1}^{2}.
\end{align*}
Using the Gronwall's inequality with Lemma \ref{L2.1}, we arrive at
\begin{align*}
\frac{\nu }{2}\norm{e^{\alpha t}w_{x}}^{2}+e^{2\alpha t}E(t)+\int_{0}^{t}e^{2\alpha s}\norm{w_{t}(s)}^{2}ds\leq C\norm{w_{0}}_{1}^{2} e^{C\norm{w_{0}}^{2}}.
\end{align*}
This completes the proof after multiplying by \(e^{-2\alpha t}\).
\end{proof}
\begin{lemma}
	\label{L2.3}
	Suppose that the  assumptions \((A1)\) and \((A2)\) are true. Then, there holds
	\begin{align*}
		\norm{w_{x}}^{2}&+\rho\norm{z_{x}}^{2}+\frac{1}{\nu}E(t)+e^{-2\alpha t}\int_{0}^{t}e^{2\alpha s}\big(\nu\norm{w_{xx}(s)}^{2}+\rho\delta \norm{z_{x}(s)}^{2}\big)ds\leq Ce^{-2\alpha t}\norm{w_{0}}_{1}^{2} e^{C\norm{w_{0}}_{1}^{2}}.
	\end{align*}
\end{lemma}
\begin{proof}
 Forming the \(L^{2}\)-inner product between \eqref{eq:4.1} and \(-w_{xx}\), we obtain
 \begin{align}
 	\label{2.12}
 	\nonumber\frac{1}{2}\frac{d}{dt}\norm{w_{x}}^{2}+\nu \norm{w_{xx}}^{2}-w_{t}(1,t)w_{x}(1,t)&+w_{t}(0,t)w_{x}(0,t)
    \\&=(ww_{x}, w_{xx})+w_{d}(w_{x},w_{xx})-\rho(z_{x}, w_{xx}).
 \end{align}
	Differentiating \eqref{eq:4.2} with respect to \(x\) and taking  \(L^{2}\)-inner product with \(z_{x}\) yields
	\begin{align}
	\label{2.13}
	\frac{1}{2}\frac{d}{dt}\norm{z_{x}}^{2}+ \delta \norm{z_{x}}= (z_{x}, w_{xx}).
	\end{align}
Also, we can write from \eqref{eq:4.2}
\begin{align}
	\label{2.14}
	z(1,t)w_{t}(1,t)-z(0,t)w_{t}(0,t)=(z_{x},w_{t})+ \delta(z,w_{x})+\frac{d}{dt}(z,w_{x})-\norm{w_{x}}^{2}.
\end{align}
Therefore, from \eqref{2.12}, \eqref{2.13} and \eqref{2.14}, it follows that
\begin{align*}
	\frac{1}{2}\frac{d}{dt}\Big(\norm{w_{x}}^{2}&+\rho\norm{z_{x}}^{2}+\frac{1}{\nu}E(t)\Big)+\nu\norm{w_{xx}}^{2}+\rho\delta\norm{z_{x}}^{2}\\&= \frac{\rho}{\nu}\norm{w_{x}}^{2}- \frac{\rho}{\nu}\frac{d}{dt}(z,w_{x})- \frac{\rho}{\nu}(z_{x},w_{t})- \frac{\rho \delta}{\nu}(z,w_{x})+(ww_{x},w_{xx})+w_{d}(w_{x},w_{xx}).
\end{align*}
Note that
\begin{align*}
	(ww_{x},w_{xx})&\leq C\norm{w_{x}}^{2}\norm{w}_{1}^{2}+\frac{\nu}{4}\norm{w_{xx}}^{2},\\
	w_{d}(w_{x},w_{xx})&\leq C\norm{w_{x}}^{2}+ \frac{\nu}{4}\norm{w_{xx}}^{2}.
\end{align*}
Using the Cauchy-Schwarz inequality and Young's inequality, we arrive at
\begin{align*}
	\frac{1}{2}\frac{d}{dt}\Big(\norm{w_{x}}^{2}&+\rho\norm{z_{x}}^{2}+\frac{1}{\nu}E(t)\Big)+\frac{\nu}{2}\norm{w_{xx}}^{2}+\frac{\rho\delta}{2}\norm{z_{x}}^{2}\\&\leq C\norm{w_{x}}^{2}+ C\norm{w_{t}}^{2}+C\norm{z}^{2}-\frac{\rho}{\nu}\frac{d}{dt}(z,w_{x})+C\norm{w_{x}}^{2}\norm{w}_{1}^{2}.
\end{align*}
Multiplying by \(2e^{2\alpha t}\) to the above equation yields
\begin{align*}
\frac{d}{dt}\Big(\norm{e^{\alpha t}w_{x}}^{2}&+\rho\norm{e^{\alpha t}z_{x}}^{2}+\frac{1}{\nu}e^{2\alpha t}E(t)\Big)+\nu e^{2\alpha t}\norm{w_{xx}}^{2}+\rho\delta e^{2\alpha t}\norm{z_{x}}^{2}\\&\leq Ce^{2\alpha t}\norm{w_{x}}^{2}+ Ce^{2\alpha t}\norm{w_{t}}^{2}+Ce^{2\alpha t}\norm{z}^{2}-\frac{\rho}{\nu}\frac{d}{dt}\Big(e^{2\alpha t}(z,w_{x})\Big)+\frac{2\alpha \rho}{\nu}e^{2\alpha t}(z,w_{x}) \\&\quad +Ce^{2\alpha t}\norm{w}_{1}^{2}\Big(\norm{w_{x}}^{2}+\rho\norm{z_{x}}^{2}+\frac{1}{\nu}E(t) \Big)\\&\quad + 2\alpha \Big(\norm{e^{2 t}w_{x}}^{2}+\rho\norm{e^{\alpha t}z_{x}}^{2}+\frac{1}{\nu}e^{2\alpha t}E(t)\Big),
\\& \leq Ce^{2\alpha t}\norm{w_{x}}^{2}+ Ce^{2\alpha t}\norm{w_{t}}^{2}+Ce^{2\alpha t}\norm{z}^{2}-\frac{\rho}{\nu}\frac{d}{dt}\Big(e^{2\alpha t}(z,w_{x})\Big)
\\& \quad + Ce^{2\alpha t}\norm{w}_{1}^{2}\Big(\norm{w_{x}}^{2}+\rho\norm{z_{x}}^{2}+\frac{1}{\nu}E(t) \Big).
\end{align*}
Integrating with respect to time over \([0,t]\) and using Lemmas \ref{L2.1}-\ref{L2.2}, we can write
\begin{align*}
	\Big(\norm{e^{\alpha t}w_{x}}^{2}&+\rho\norm{e^{\alpha t}z_{x}}^{2}+\frac{1}{\nu}e^{2\alpha t}E(t)\Big)+\int_{0}^{t}e^{2\alpha s}\big(\nu\norm{w_{xx}(s)}^{2}+\rho\delta \norm{z_{x}(s)}^{2}\big)ds\\&\leq C\norm{w_{0}}_{1}^{2}-e^{2\alpha t}\frac{\rho}{\nu}(z,w_{x})+  C\int_{0}^{t}e^{2\alpha s}\norm{w(s)}_{1}^{2}\Big(\norm{w_{x}(s)}^{2}+\rho\norm{z_{x}(s)}^{2}+\frac{1}{\nu}E(s) \Big)ds.
\end{align*}
Again, using the Cauchy-Schwarz inequality, Young's inequality and Lemmas \ref{L2.1}-\ref{L2.2} with  Gronwall's inequality, we observe that
\begin{align*}
	\norm{e^{\alpha t}w_{x}}^{2}&+\rho\norm{e^{\alpha t}z_{x}}^{2}+\frac{1}{\nu}e^{2\alpha t}E(t)+\int_{0}^{t}e^{2\alpha s}\big(\nu\norm{w_{xx}(s)}^{2}+\rho\delta \norm{z_{x}(s)}^{2}\big)ds\leq C\norm{w_{0}}_{1}^{2} e^{C\norm{w_{0}}_{1}^{2}}.
\end{align*}
The proof is completed after multiplying by \(e^{-2\alpha t}\).
\end{proof}
\begin{remark}
	\label{r2.1}
	From \eqref{eq:4.2}, we arrive at
	\begin{align*}
		\norm{z_{t}}^{2}\leq 2\delta\norm{z}^{2}+2\norm{w_{x}}^{2}.
	\end{align*}
Using Lemmas \ref{L2.1} and \ref{L2.2}, we have
\begin{align*}
	e^{2\alpha t}\norm{z_{t}}^{2}\leq C\norm{w_{0}}_{1}^{2}e^{C\norm{w_{0}}_{1}^{2}}.
\end{align*}
\end{remark}
\begin{lemma}
	\label{L2.4}
	Under the assumptions \((A1)\) and \((A2)\), the following holds
	\begin{align*}
		\norm{w_{t}}^{2}+\rho \norm{z_{t}}^{2}+e^{-2\alpha t}\int_{0}^{t}e^{2\alpha s}\Big(\nu\norm{w_{xt}(s)}^{2}&+2\rho\delta\norm{z_{t}(s)}^{2}+E_{1}(s)\Big)ds\\&\leq C(\norm{w_{0}}_{1})\norm{w_{0}}_{2}^{2}e^{-2\alpha t},
	\end{align*}
where \(E_{1}(t)=\sum \limits_{i=0}^{1}\Big((c_{i}+w_{d})w_{t}^{2}(i,t)+\frac{4}{3c_{i}}w_{t}^{2}(i,t) w^{2}(i,t)\Big)\).
\end{lemma}
\begin{proof}
	Differentiating \eqref{1.13} with respect to time $t$ and setting \(\phi=w_{t}\) yields
	\begin{align}
		\label{2.15}
		\frac{1}{2}\frac{d}{dt}\norm{w_{t}}^{2}+\nu\norm{w_{xt}}^{2}&+\frac{w_{d}}{2}\big(w_{t}^{2}(1,t)-w_{t}^{2}(0,t)\big)+ \sum_{i=0}^{1}\Big((c_{i}+w_{d})w_{t}^{2}(i,t)+\frac{2}{3c_{i}}w_{t}^{2}(i,t)w^{2}(i,t)\Big)
		\nonumber\\&+\rho z_{t}(1,t)w_{t}(1,t)-\rho z_{t}(0,t)w_{t}(0,t)=-(w_{t}w_{x}+ww_{xt}, w_{t})+\rho(z_{xt},w_{t}).
	\end{align}
Also, differentiating \eqref{1.14}  with respect to time and forming $L^{2}$-inner product with \(z_{t}\) gives
\begin{align}
		\label{2.16}
	\frac{1}{2}\frac{d}{dt}\norm{z_{t}}^{2}+\delta\norm{z_{t}}^{2}=w_{t}(1,t)z_{t}(1,t)-w_{t}(0,t)z_{t}(0,t)-(z_{xt},w_{t}).
\end{align}
Multiplying \eqref{2.16} by \(\rho\) and using \eqref{2.15}, we get
\begin{align}
	\label{2.17}		\frac{1}{2}\frac{d}{dt}\Big(\norm{w_{t}}^{2}+\rho\norm{z_{t}}^{2}\Big)&+\nu\norm{w_{xt}}^{2}+\rho\delta\norm{z_{t}}^{2}+\Big((c_{0}+\frac{w_{d}}{2})w_{t}^{2}(0,t)+\frac{2}{3c_{0}}w_{t}^{2}(0,t)w^{2}(0,t)\Big)
	\nonumber\\&+ \Big((c_{1}+\frac{3w_{d}}{2})w_{t}^{2}(1,t)+\frac{2}{3c_{1}}w_{t}^{2}(1,t)w^{2}(1,t)\Big)=-(w_{t}w_{x}+w w_{xt}, w_{t}).
\end{align}
The term on the right hand side is estimated by
\begin{align*}
	-(w_{t}w_{x}+w w_{xt}, w_{t})&\leq \norm{w_{t}}_{\infty}\norm{w_{x}}\norm{w_{t}}+\norm{w}_{\infty}\norm{w_{xt}}\norm{w_{t}},
	\\&\leq(|w_{t}(0,t)|+\norm{w_{xt}}) \norm{w_{x}}\norm{w_{t}}+ C\norm{w}_{1}\norm{w_{xt}}\norm{w_{t}},
	\\&\leq \frac{c_{0}}{2}w_{t}^{2}(0,t)+\frac{\nu}{2}\norm{w_{xt}}^{2}+ C\norm{w}_{1}^{2}\norm{w_{t}}^{2}.
\end{align*}
 From \eqref{2.17}, it follows that
\begin{align*}
   \frac{1}{2}\frac{d}{dt}\Big(\norm{w_{t}}^{2}+\rho\norm{z_{t}}^{2}\Big)&+\frac{\nu}{2}\norm{w_{xt}}^{2}+\rho\delta\norm{z_{t}}^{2}+\frac{1}{2}\Big((c_{0}+w_{d})w_{t}^{2}(0,t)+\frac{4}{3c_{0}}w_{t}^{2}(0,t)w^{2}(0,t)\Big)
   \nonumber\\&+\frac{1}{2} \Big((c_{1}+w_{d})w_{t}^{2}(1,t)+\frac{4}{3c_{1}}w_{t}^{2}(1,t)w^{2}(1,t)\Big)\leq C\norm{w}_{1}^{2}\norm{w_{t}}^{2}.
\end{align*}
Multiplying by \(2e^{2\alpha t}\) to the above inequality and integrating with respect to time from \(0\) to \(t\), and then using Lemmas \ref{L2.1}-\ref{L2.2} together with Remark \ref{r2.1}, we obtain
\begin{align}
	\label{2.18}
	e^{2\alpha t} \norm{w_{t}}^{2}+\rho e^{2\alpha t}\norm{z_{t}}^{2}+\int_{0}^{t}e^{2\alpha s}\Big(\nu\norm{w_{xt}(s)}^{2}&+2\rho\delta\norm{z_{t}(s)}^{2}+E_{1}(s)\Big)ds
	\nonumber\\&\leq \norm{w_{t}(0)}^{2}+\rho\norm{z_{t}(0)} +C\norm{w_{0}}_{1}^{2},
\end{align}
where $E_{1}(t)=\sum\limits_{i=0}^{1}\Big((c_{i}+w_{d})w_{t}^{2}(i,t)+\frac{4}{3c_{i}}w_{t}^{2}(i,t) w^{2}(i,t)\Big)$.

From \eqref{eq:4.1}, we observe that
\begin{align*}
	\norm{w_{t}}^{2}\leq C(\norm{w_{xx}}^{2}+\norm{w_{x}}^{2}+\norm{z_{x}}^{2}+\norm{w}_{1}^{2}\norm{w_{x}}^{2}).
\end{align*}
 Substituting \(t=0\) to the above inequality yields
\begin{align*}
	\norm{w_{t}(0)}^{2}\leq C(\norm{w_{0}}_{2}^{2}+\norm{w_{0}}_{1}^{4}).
\end{align*}
Using \eqref{eq:4.2}, we arrive at
\begin{align*}
	\norm{z_{t}(0)}^{2}\leq 2\norm{w_{0}}_{1}^{2}.
\end{align*}
Consequently, the proof is completed after substituting the value of \(\norm{w_{t}(0)}^{2} \) and multiplying by \(e^{-2\alpha t}\) in the resulting inequality.
\end{proof}
\begin{lemma}
	\label{L2.5}
	Under assumptions \((A1)\) and \((A2)\), there exists a positive constant $C$ such that
	\begin{align*}
			 \norm{w_{xx}}^{2}\leq C(\norm{w_{0}}_{1})e^{-2\alpha t} \norm{w_{0}}_{2}^{2}e^{C\norm{w_{0}}^{2}}.
	\end{align*}
\end{lemma}
\begin{proof}
	From \eqref{eq:4.1}, we have
	\begin{align*}
	 \nu w_{xx}(x,t)=w_t(x,t) + w(x,t)w_x(x,t) + w_dw_x(x,t)- \rho z_x(x,t).
	\end{align*}
A use of the triangle inequality gives
\begin{align*}
	\nu \norm{w_{xx}}^{2}\leq C(\norm{w_{t}}^{2}+ \norm{w_{x}}^{2}+ \norm{w}_{1}^{1}\norm{w_{x}}^{2}+\norm{z_{x}}^{2}).
\end{align*}
Using Lemmas \ref{L2.1}, \ref{L2.3}, and \ref{L2.4}, we obtain
\begin{align*}
	\nu \norm{w_{xx}}^{2}\leq Ce^{-2\alpha t}\big( \norm{w_{0}}_{2}^{2}+\norm{w_{0}}_{1}^{2}\big)e^{C\norm{w_{0}}^{2}}.
\end{align*}
\end{proof}
\begin{lemma}
	\label{L2.6}
	Let the assumptions \((A1)\)- \((A2)\) hold. Then, there exists a positive constant $C$ such that
	\begin{align*}
	\norm{w_{xt}}^{2}+\rho\norm{z_{xt}}^{2}+ \frac{2}{\nu}E_{2}(t)&+2\nu e^{-2\alpha t} \int_{0}^{t}e^{2\alpha s}\norm{w_{xxt}(s)}^{2}ds+ 4\rho \delta e^{-2\alpha t} \int_{0}^{t} e^{2\alpha s}\norm{z_{xt}(s)}^{2} ds
	\\&\leq  e^{-2\alpha t}C(\norm{w_{0}}_{1})\norm{w_{0}}_{3}^{2},
	\end{align*}
where \(E_{2}(t)=\sum\limits_{i=0}^{1}\Big((c_{i}+w_{d})w_{t}^{2}(i,t)+\frac{2}{3c_{i}}w_{t}^{2}(i,t) w^{2}(i,t)\Big)\).
\end{lemma}
\begin{proof}
   Differentiating \eqref{eq:4.1} with respect to time and then forming the \(L^{2}\)-inner product with \(-w_{xxt}\) yields
   \begin{align}
   	\label{2.141}
   	\frac{1}{2}\frac{d}{dt}\norm{w_{xt}}^{2}+ \nu \norm{w_{xxt}}^{2}&+ \frac{1}{\nu}\sum_{i=0}^{1}\Big((c_{i}+w_{d})w_{t}(i,t)w_{tt}(i,t)+\frac{2}{3c_{i}}w^{2}(i,t)w_{t}(i,t)w_{tt}(i,t)\Big)
   \nonumber\\& +\frac{\rho}{\nu}z_{t}(1,t)w_{tt}(1,t)-\frac{\rho}{\nu}z_{t}(0,t)w_{tt}(0,t)
   \nonumber\\&	=-\rho (z_{xt},w_{xxt})+(w_{t}w_{x}+ww_{xt}, w_{xxt})+w_{d}(w_{xt}, w_{xxt}).
   \end{align}
After differentiating \eqref{eq:4.2} with respect to $t$ and \(x\), we can write
\begin{align}
	\label{2.151}
	\frac{1}{2}\frac{d}{dt}\norm{z_{xt}}^{2}+\delta\norm{z_{xt}}^{2}=(z_{xt},w_{xxt}),
\end{align}
and 
\begin{align}
	\label{2.161}
	z_{t}(1)w_{tt}(1,t)-z_{t}(0)w_{tt}(0,t)&=\frac{d}{dt}(z_{tt},z_{t})-\norm{z_{tt}}^{2}+\delta(z_{tt},z_{t})+\frac{d}{dt}(w_{t},z_{xt})
	\nonumber\\&\quad -\delta^{2}(z_{x},w_{t})+\delta(w_{t},w_{xx})-(w_{t},w_{xxt}).
\end{align}
Multiplying \eqref{2.151} by \(\rho\), it follows  from \eqref{2.141} and \eqref{2.161} 
\begin{align*}
	\frac{1}{2}\frac{d}{dt}\Big(\norm{w_{xt}}^{2}+\rho\norm{z_{xt}}^{2}+ \frac{1}{\nu}E_{2}(t)\Big)&+ \nu \norm{w_{xxt}}^{2}+ \rho \delta\norm{z_{xt}}^{2} 
	\nonumber\\&=\sum_{i=0}^{1}\frac{1}{3c_{i}\nu}\Big(w(i,t)w_{t}^{3}(i,t)\Big)+(w_{t}w_{x}+ww_{xt}, w_{xxt})
	\nonumber\\& \quad +w_{d}(w_{xt}, w_{xxt})- \frac{\rho}{\nu}\frac{d}{dt}(z_{tt},z_{t})+\frac{\rho}{\nu}\norm{z_{tt}}^{2}-\frac{\rho \delta}{\nu}(z_{tt},z_{t})\nonumber\\&\quad-\frac{\rho}{\nu}\frac{d}{dt}(w_{t},z_{xt})
	 +\frac{\rho \delta^{2}}{\nu}(z_{x},w_{t})-\frac{\rho \delta}{\nu}(w_{t},w_{xx})+\frac{\rho }{\nu}(w_{t},w_{xxt}).
\end{align*}
Using Young's inequality to the above equation, we arrive at
\begin{align*}
	\frac{1}{2}\frac{d}{dt}\Big(\norm{w_{xt}}^{2}+\rho\norm{z_{xt}}^{2}+ \frac{1}{\nu}E_{2}(t)\Big)&+ \frac{\nu}{2} \norm{w_{xxt}}^{2}+ \rho \delta\norm{z_{xt}}^{2} 
	\nonumber\\&\leq \sum_{i=0}^{1}\frac{1}{3c_{i}\nu}\Big(w^{2}(i,t)w_{t}^{2}(i,t)+w_{t}^{4}(i,t)\Big)+C\norm{w_{xt}}^{2}\norm{|w|}^{2}
	\nonumber\\&\quad +Cw_{t}^{2}(0,t)\norm{w_{xt}}^{2}+C\norm{w_{xt}}^{2}
	- \frac{\rho}{\nu}\frac{d}{dt}(z_{tt},z_{t})-\frac{\rho}{\nu}\frac{d}{dt}(w_{t},z_{xt})
	\nonumber\\&\quad+C\norm{z_{tt}}^{2}+ C\Big(\norm{z_{t}}^{2}+\norm{w_{t}}^{2}+\norm{w_{xx}}^{2}+\norm{z_{x}}^{2}\Big).	
\end{align*}
Multiplying by \(2e^{2\alpha t}\) to the above inequality and integrating with respect to time over \([0,t]\), we observe that
\begin{align}
	\label{2.171}
 e^{2\alpha t}\Big(\norm{w_{xt}}^{2}&+\rho\norm{z_{xt}}^{2}+ \frac{1}{\nu}E_{2}(t)\Big)+\nu \int_{0}^{t}e^{2\alpha s}\norm{w_{xxt}(s)}^{2}ds+ 2\rho \delta \int_{0}^{t} e^{2\alpha s}\norm{z_{xt}(s)}^{2} ds
 \nonumber\\&\leq 2\alpha \int_{0}^{t} e^{2\alpha s}\Big(\norm{w_{xt}(s)}^{2}+\rho\norm{z_{xt}(s)}^{2}+ \frac{1}{\nu}E_{2}(s)\Big)ds+C\int_{0}^{t} e^{2\alpha s}\norm{w_{xt}(s)}^{2}\norm{|w(s)|}^{2}ds
 \nonumber\\& \quad + \int_{0}^{t} e^{2\alpha s}\sum_{i=0}^{1}\frac{1}{3c_{i}\nu}\Big(w^{2}(i,s)w_{t}^{2}(i,s)+w_{t}^{4}(i,s)\Big)ds +C \int_{0}^{t} e^{2\alpha s}w_{t}^{2}(0,s)\norm{w_{xt}(s)}^{2}ds
 \nonumber\\&\quad + C\int_{0}^{t} e^{2\alpha s}\Big(\norm{w_{xt}(s)}^{2}+\norm{z_{tt}(s)}^{2}+\norm{z_{t}(s)}^{2}+\norm{w_{t}(s)}^{2}+\norm{w_{xx}(s)}^{2}+\norm{z_{x}(s)}^{2}\Big)ds
 \nonumber\\&\quad -\frac{\rho e^{2\alpha t}}{\nu}(z_{tt},z_{t})-\frac{\rho e^{2\alpha t}}{\nu}(w_{t},z_{xt})+C\int_{0}^{t} e^{2\alpha s}\norm{z_{xt}(s)}^{2}ds+ \norm{w_{xt}(0)}^{2}+\norm{z_{xt}(0)}^{2}.
\end{align}
Note that \[\norm{z_{tt}}^{2}\leq 2\delta^{2}\norm{z_{t}}^{2}+2\norm{w_{xt}}^{2}.
\]
From \eqref{eq:4.1} and \eqref{eq:4.2}, we can write after substituting \(t=0\)
\[\norm{w_{xt}(0)}^{2}+\norm{z_{xt}(0)}^{2}\leq C\Big(\norm{w_{0}}_{3}^{2}+\norm{w_{0}}_{1}^{2}\norm{w_{0}}_{2}^{2}+\norm{w_{0}}_{1}^{2} \Big).
\]
Using the Cauchy-Schwarz inequality and Young's inequality, we have from \eqref{2.171}
\begin{align*}
	e^{2\alpha t}\Big(\frac{1}{2}\norm{w_{xt}}^{2}&+\frac{\rho}{2}\norm{z_{xt}}^{2}+ \frac{1}{\nu}E_{2}(t)\Big)+\nu \int_{0}^{t}e^{2\alpha s}\norm{w_{xxt}(s)}^{2}ds+ 2\rho \delta \int_{0}^{t} e^{2\alpha s}\norm{z_{xt}(s)}^{2} ds
	\nonumber\\&\leq C \int_{0}^{t} e^{2\alpha s}\Big(\norm{w_{xt}(s)}^{2}+\rho\norm{z_{xt}(s)}^{2}+ \frac{1}{\nu}E_{2}(s)\Big)\Big(\norm{|w(s)|}^{2}+ w_{t}^{2}(0,s)+w_{t}^{2}(1,s)\Big)ds
	\\&\quad +C \int_{0}^{t} e^{2\alpha s}w_{t}^{2}(0,s)\norm{w_{xt}(s)}^{2}ds+ C\int_{0}^{t} e^{2\alpha s}\Big(\norm{w_{xt}(s)}^{2}+\norm{z_{t}(s)}^{2}+\norm{w_{t}(s)}^{2}\Big)ds
	\\& \quad +C\int_{0}^{t} e^{2\alpha s}\Big(\norm{w_{xx}(s)}^{2}+\norm{z_{x}(s)}^{2}+\norm{z_{xt}(s)}^{2}\Big)ds+C\Big(\norm{z_{t}}^{2} +\norm{w_{t}}^{2}\Big)
	\\&\quad + C\Big(\norm{w_{0}}_{3}^{2}+\norm{w_{0}}_{1}^{2}\norm{w_{0}}_{2}^{2}\Big).
\end{align*}
With the help of Gronwall's inequality and Lemmas \ref{L2.2}-\ref{L2.4}, the proof is completed after multiplying by \(2e^{-2\alpha t}\).
\end{proof}
The proof of the following theorem follows from Lemmas \ref{L2.1}-\ref{L2.6}.
\begin{theorem}
\label{thm2}
    Suppose that the assumptions \((A1)\)- \((A2)\) are true. Then, the following holds
   \begin{align*}   \norm{w}_{2}^{2}+\norm{w_{t}}_{1}^{2}+\norm{z}^{2}&+E(t)+E_{2}(t)
   +e^{-2\alpha t}\int_{0}^{t}e^{2\alpha s}\Big(\norm{|w(s)|}^{2}+\norm{w_{t}(s)}_{2}^{2}+\norm{w(s)}_{2}^{2}\Big)ds\\&+e^{-2\alpha t}\int_{0}^{t}e^{2\alpha s}\Big(\norm{z(s)}^{2}+E_{1}(s)+w^{4}(1,s)+w^{4}(0,s)\Big)ds\leq Ce^{-2\alpha t}\norm{w_{0}}_{3}^{2}.
   \end{align*} 
\end{theorem}
\subsection{Adaptive Controller.}
In this section, we discuss adaptive feedback control using the Lyapunov approach, when the diffusion coefficient \(\nu \) is unknown. In the previous section, global stabilization is achieved for known $\nu$. The proof of adaptive feedback control for the viscous Burgers' equation $(\rho=0)$ is given in \cite{liu2001adaptive}.

We consider the Lyapunov functional
\begin{align}
	\label{3.1}
	V=\norm{w}^{2}+\rho\norm{z}^{2}+\frac{\nu}{\gamma}\big(\eta_{0}-\frac{1}{2\nu}\big)^{2}+ \frac{\nu}{\gamma}\big(\eta_{1}-\frac{1}{2\nu}\big)^{2},
\end{align}
where \(\gamma\) is a positive constant. In addition, $\eta_{0}$ and $\eta_{1}$ are used as estimates of $\frac{1}{2\nu}$.

Differentiating \eqref{3.1} with respect to time, we obtain
\begin{align}
	\dot{V}=2\int_{0}^{1}\big(ww_{t}+\rho zz_{t}\big)dx+ \frac{2\nu}{\gamma}\big(\eta_{0}-\frac{1}{2\nu}\big)\dot{\eta_{0}}+ \frac{2\nu}{\gamma}\big(\eta_{1}-\frac{1}{2\nu}\big)\dot{\eta_{1}}.
\end{align}
Using integration by parts, we get from \eqref{eq:4.1}-\eqref{eq:4.2}
\begin{align*}
	\dot{V}&=2\nu w(1,t)w_{x}(1,t)-2\nu w(0,t)w_{x}(0,t)-2\nu\norm{w_{x}}^{2}-2\rho \delta\norm{z}^{2}-\frac{w_{d}}{2}\Big(w^{2}(1,t)-w^{2}(0,t)\Big)
	\\&\quad -\frac{1}{3}\Big(w^{3}(1,t)-w^{3}(0,t)\Big)+2\rho z(1,t)w(1,t)-2\rho z(0,t)w(0,t)+ \frac{2\nu}{\gamma}\big(\eta_{0}-\frac{1}{2\nu}\big)\dot{\eta_{0}}+ \frac{2\nu}{\gamma}\big(\eta_{1}-\frac{1}{2\nu}\big)\dot{\eta_{1}}.
\end{align*}
Using the Cauchy-Schwarz inequality and Young's inequality, we observe that
\begin{align*}
	\dot{V}&\leq -2\nu\norm{w_{x}}^{2}-2\rho \delta\norm{z}^{2}- 2\nu w(0,t)w_{x}(0,t)+ w(0,t) \Big((c_{0}+w_{d})w(0,t)+\frac{2}{9c_{0}}w^{3}(0,t)\Big)
	\\&\quad +2\nu w(1,t)w_{x}(1,t)+ w(1,t)\Big((c_{1}+w_{d})w(1,t)+\frac{2}{9c_{1}}w^{3}(1,t)\Big)+2\rho z(1,t)w(1,t)
	\\&\quad -2\rho z(0,t)w(0,t)
	 + \frac{2\nu}{\gamma}\big(\eta_{0}-\frac{1}{2\nu}\big)\dot{\eta_{0}}+ \frac{2\nu}{\gamma}\big(\eta_{1}-\frac{1}{2\nu}\big)\dot{\eta_{1}}-\frac{w_d}{2}w^2(0,t)-\frac{c_{0}}{2} w^2(0,t)
	 \\&\quad-\frac{3w_d}{2}w^2(1,t) - \frac{c_{1}}{2} w^{2}(1,t).
\end{align*}
Therefore, it follows that
\begin{align*}
     \dot{V}&\leq -2\nu\norm{w_{x}}^{2}-2\rho \delta\norm{z}^{2}- 2\nu w(0,t)w_{x}(0,t)+2\nu w(1,t)w_{x}(1,t)
     \\&\quad + 2\nu(\frac{1}{2\nu}-\eta_{0}+\eta_{0})\Big((c_{0}+w_{d})w^{2}(0,t)+\frac{2}{9c_{0}}w^{4}(0,t)-2\rho z(0,t)w(0,t)\Big)
     \\&\quad + 2\nu(\frac{1}{2\nu}-\eta_{1}+\eta_{1})\Big((c_{1}+w_{d})w^{2}(1,t)+\frac{2}{9c_{1}}w^{4}(1,t)+2\rho z(1,t)w(1,t)\Big)
     \\&\quad
     + \frac{2\nu}{\gamma}\big(\eta_{0}-\frac{1}{2\nu}\big)\dot{\eta_{0}}+ \frac{2\nu}{\gamma}\big(\eta_{1}-\frac{1}{2\nu}\big)\dot{\eta_{1}}-\frac{w_d}{2}w^2(0,t)-\frac{c_{0}}{2} w^2(0,t)
     \\&\quad-\frac{3w_d}{2}w^2(1,t) - \frac{c_{1}}{2} w^{2}(1,t).  
\end{align*}
Rewrite to the above inequality
\begin{align}
	\label{3.3}
	\dot{V}&\leq -2\nu\norm{w_{x}}^{2}-2\rho \delta\norm{z}^{2}-\frac{w_d}{2}w^2(0,t)-\frac{c_{0}}{2} w^2(0,t)
	-\frac{3w_d}{2}w^2(1,t) - \frac{c_{1}}{2} w^{2}(1,t)
	\nonumber\\&\quad-2\nu w(0,t)\Big( V_{0}(t)-\eta_{0}\big((c_{0}+w_{d})w(0,t)+\frac{2}{9c_{0}}w^{3}(0,t)-2\rho z(0,t)\big) \Big)
	\nonumber\\&\quad +2\nu w(1,t)\Big( V_{1}(t)+\eta_{1}\big((c_{1}+w_{d})w(1,t)+\frac{2}{9c_{1}}w^{3}(1,t)+2\rho z(1,t)\big) \Big)
   \nonumber\\&\quad + 2\nu(\frac{1}{2\nu}-\eta_{0})\Big((c_{0}+w_{d})w^{2}(0,t)+\frac{2}{9c_{0}}w^{4}(0,t)-2\rho z(0,t)w(0,t)-\frac{\dot{\eta_{0}}}{\gamma}\Big)
	\nonumber\\&\quad + 2\nu(\frac{1}{2\nu}-\eta_{1})\Big((c_{1}+w_{d})w^{2}(1,t)+\frac{2}{9c_{1}}w^{4}(1,t)+2\rho z(1,t)w(1,t)-\frac{\dot{\eta_{1}}}{\gamma}\Big).
\end{align}
Now, we choose the adaptive feedback control
\begin{align*}
	 V_{0}(t)=\eta_{0}\big((c_{0}+w_{d})w(0,t)+\frac{2}{9c_{0}}w^{3}(0,t)-2\rho z(0,t)\big),\\
	 V_{1}(t)=-\eta_{1}\big((c_{1}+w_{d})w(1,t)+\frac{2}{9c_{1}}w^{3}(1,t)+2\rho z(1,t)\big),\\
	\dot{\eta_{0}}=\gamma \Big((c_{0}+w_{d})w^{2}(0,t)+\frac{2}{9c_{0}}w^{4}(0,t)-2\rho z(0,t)w(0,t)\Big),\\
	\dot{\eta_{1}}=\gamma\Big((c_{1}+w_{d})w^{2}(1,t)+\frac{2}{9c_{1}}w^{4}(1,t)+2\rho z(1,t)w(1,t)\Big),
\end{align*}
with $\eta_{0}(0)=\eta_{0}^{0}$ and $\eta_{1}(0)=\eta_{1}^{0}.$
Hence, with this control, we arrive at from \eqref{3.3}
\begin{align*}
	\dot{V}&\leq -2\nu\norm{w_{x}}^{2}-2\rho \delta\norm{z}^{2}-\frac{w_d}{2}w^2(0,t)-\frac{c_{0}}{2} w^2(0,t)
	-\frac{3w_d}{2}w^2(1,t) - \frac{c_{1}}{2} w^{2}(1,t),
	\\&\leq -\min\{2\nu, 2\delta, \frac{(c_{0}+w_{d})}{2}, \frac{(c_{1}+3w_{d})}{2}  \}V.
\end{align*}
where \(\nu, \delta,\) and \(\rho\) are positive constants. Therefore, we obtain stability in the $L^{2}$-norm.
\section{Existence and Uniqueness.}
\label{3}
This section provides the existence and uniqueness of the weak solution to problem \eqref{eq:4.1}-\eqref{eq:4.8} following \cite{MR4019979}. We use the Faedo-Galerkin approximation with boundary conditions
\begin{align}
	w_{x}(0,t)-k_{0} w(0,t)+ \frac{\rho }{\nu}z(0,t)= \frac{2}{9c_{0} \nu}w^{3}(0,t),
\end{align}
and \begin{align}
	w_{x}(1,t)+k_{1} w(1,t)- \frac{\rho }{\nu}z(1,t)= -\frac{2}{9c_{1} \nu}w^{3}(1,t),
\end{align}
where \(k_{i}=\frac{(c_{i}+w_{d})}{\nu}, i=0,1\). Let \(A:= -\frac{d^{2}}{dx^{2}}\) be a linear operator on \(L^{2}(0,1)\) with dense domain \(D(A)\), where 
 \[ D(A)=\{v\in H^{2}(0,1): v_{x}(0)-k_{0} v(0)+\frac{\rho }{\nu}z(0)=0, \ \text{and}  \ v_{x}(1)+k_{1} v(1)-\frac{\rho }{\nu}z(1)=0 \}.
 \]
Moreover, \(k_{0}+k_{1}>0\) and \(A\) is a self-adjoint positive definite with compact inverse,  since \(H^{1}(0,1)\hookrightarrow L^{2}(0,1)\) is compact. Then by spectral theorem, there exists a sequence of eigenvalues of \(A\), \(0<\lambda_{1}<\lambda_{2}<\ldots<\lambda_{k}<\ldots\) and \(\lim_{k\rightarrow\infty} \lambda_{k}=\infty\). Further, \((\phi_{k})_{k=1}^{\infty}\) forms an orthonormal basis of \(L^{2}(0,1)\), which consists of the eigenfunction of \(A\).

Let \((\phi_{k})_{k=1}^{\infty}\) be an orthogonal basis of \(H^{1}(0,1)\). Let \(m\) be a positive integer, and suppose that \(V_{m}\) is the space spanned by \(\{\phi_{1},\ldots,\phi_{m}\}\). Now for a positive integer \(m\), we seek a function \(w^{m}:[0,T]\rightarrow H^{1}(0,1)\) such that \[w^{m}(t)=\sum_{i=1}^{m}\alpha_{i}^{m}(t)\phi_{i}(x).\] The coefficients \(\alpha_{i}^{m}(t)\), \(0\leq t\leq T\) are chosen such that \(\alpha_{i}^{m}(0)=(w_{0}, \phi_{i}), i=1,2,\ldots,m\) and \(w^{m}\) satisfying 
\begin{align}
	\label{4.31}
 	(w_{t}^{m},\phi_{i}) &+\nu (w_{x}^{m},\phi_{ix})+w_{d}(w_{x}^{m},\phi_{i}) +(w^{m}w_{x}^{m},\phi_{i})+ \bigg((c_{0}+w_{d})w^{m}(0,t)+  \frac{2}{9c_{0}}(w^{m})^{3}(0,t)\bigg)\phi_{i}(0)\nonumber\\&-\rho z^{m}(0)\phi_{i}(0) +\left((c_{1}+w_{d})w^{m}(1,t)+ \frac{2}{9c_{1}}(w^{m})^{3}(1,t)+\rho z^{m}(1)\right)\phi_{i}(1)=\rho (z_{x}^{m}, \phi_{i}),	
\end{align}
where \(z^{m}(x,t)= \int_{0}^{t}e^{-\delta(t-s)}w_{x}^{m}(x,s)ds\), and 
\begin{align}
	\label{4.41}
	z_{t}^{m}+\delta z^{m}=w_{x}^{m}.
\end{align}
Multiplying \eqref{4.31} by \(\alpha_{i}^{m}(t)\) and summing from \(i=1,2,\ldots,m\), we obtain from \eqref{4.31}
\begin{align*}
	\frac{1}{2}\frac{d}{dt}\norm{w^{m}}^{2} &+\nu \norm{w_{x}^{m}}^{2}+w_{d}(w_{x}^{m},w^{m}) +(w^{m}w_{x}^{m},w^{m})+ \bigg((c_{0}+w_{d})w^{m}(0,t)+  \frac{2}{9c_{0}}(w^{m})^{3}(0,t)\bigg)w^{m}(0,t)\nonumber\\&-\rho z^{m}(0)w^{m}(0,t) +\left((c_{1}+w_{d})w^{m}(1,t)+ \frac{2}{9c_{1}}(w^{m})^{3}(1,t)+\rho z^{m}(1)\right)w^{m}(1,t)=\rho (z_{x}^{m}, w^{m}).	
\end{align*}
Following the proof of Lemma \ref{L2.1} with \(\alpha=0\) yields
\begin{align*}
	\norm{w^{m}}^{2}+\rho\norm{z^{m}}^{2}&+\beta_{0} \int_{0}^{t}\left(\norm{w_{x}^{m}(s)}^{2}+\norm{z^{m}(s)}^{2}+(w^{m})^{2}(0,s)+(w^{m})^{2}(1,s) \right)ds \nonumber\\&+ \frac{1}{3c_{0}}\int_{0}^{t}(w^{m})^{4}(0,s)ds +\frac{1}{3c_{1}}\int_{0}^{t}e^{2\alpha s}(w^{m})^{4}(1,s)ds\leq C\norm{w_{0}}^{2},
\end{align*}
where $ 0<\beta_{0}= \min\{2\nu, 2\rho\delta, (c_{0}+w_{d}), (c_{1}+w_{d})\}$ and \(\norm{w^{m}(0)}\leq C\norm{w_{0}}\).
Therefore, using Lemmas \ref{L2.2}--\ref{L2.6}, we obtain the uniformly bounds of the sequences \((w_{t}^{m}), (w^{m}_{xx}), (w_{xt}^{m}), (z^{m})\),  \((E(w^{m})), (E_{1}(w^{m}))  \) and \((E_{2}(w^{m})),\) where \(E_{i}(w^{m})=E_{i}(t),\  i=0,1,2\).

As {\it{a priori}} bounds are uniformly bounded and do not depend on $m$. Hence, by using the Banach–Alaoglu theorem, there exists a subsequence \(w^{m_{n}}\) of \(w^{m}\) such that 
\begin{align}
	\label{4.51}
	\begin{cases}
		w^{m_{n}}\rightarrow w \quad \text{weakly * in} \quad  L^{\infty}(0, T; H^{2}(0,1) ),\\
		w_{t}^{m_{n}}\rightarrow w_{t} \quad \text{weakly * in} \quad L^{\infty}(0, T; H^{1}(0,1) ), \\
		z^{m_{n}}\rightarrow z \quad \text{weakly * in} \quad L^{\infty}(0, T; L^{2}(0,1) ), \\
		w^{m_{n}}\rightarrow w \quad \text{weakly in} \quad  L^{2}(0, T; H^{2}(0,1)), \\
		w_{t}^{m_{n}}\rightarrow w_{t} \quad \text{weakly in} \quad L^{2}(0, T; H^{2}(0,1)),		
	\end{cases}
\end{align}
as \(n\rightarrow \infty.\) 

Since \(H^{1}(0,1)\subset L^{2}(0,1)\) is a compact embedding \cite[Chapter 5]{MR1625845}, then from the Aubin-Lions compactness lemma, the following holds strongly
\begin{align}
	\label{4.61}
	y^{m_{n}}\rightarrow y, \ \text{in} \ L^{2}(0,T,L^{2}(0,1)),
\end{align} 
as \(n\rightarrow \infty.\)
We choose a function \(v\in C^{1}(0,T, H^{1}(0,1))\) of the form \[v(t)=\sum_{i=1}^{N}d_{i}^{m}(t)\phi_{i}(x),\] where \(d_{i}^{m}(t)\) are given smooth function, and  \(N\) is a fixed positive integer. We choose \(m\geq N\). Multiplying by  \(d_{i}^{m}(t)\) in \eqref{4.31}, summing from \(i=1\) to $N$ and integrating with respect to time from \(0\) to \(T\), we have
\begin{align}
	\label{4.71}
	\int_{0}^{T}\Big[(w_{t}^{m},v)&+\nu (w_{x}^{m},v)+w_{d}(w_{x}^{m},v) +(w^{m}w_{x}^{m},v)+ \bigg((c_{0}+w_{d})w^{m}(0,t)+  \frac{2}{9c_{0}}(w^{m})^{3}(0,t)\bigg)v(0)\nonumber\\&-\rho z^{m}(0)v(0) +\left((c_{1}+w_{d})w^{m}(1,t)+ \frac{2}{9c_{1}}(w^{m})^{3}(1,t)+\rho z^{m}(1)\right)v(1)\Big]dt=\rho\int_{0}^{T} (z_{x}^{m}, v)dt.	
\end{align}
Choosing \(m=m_{n}\) and taking \(n\rightarrow \infty\) in \eqref{4.71}, and then using \eqref{4.51} and \eqref{4.61}, we get 
\begin{align}
	\label{4.81}
	\int_{0}^{T}\Big[(w_{t},v)&+\nu (w_{x},v)+w_{d}(w_{x},v) +(ww_{x},v)+ \bigg((c_{0}+w_{d})w(0,t)+  \frac{2}{9c_{0}}w^{3}(0,t)\bigg)v(0)\nonumber\\&-\rho z(0)v(0) +\left((c_{1}+w_{d})w(1,t)+ \frac{2}{9c_{1}}w^{3}(1,t)+\rho z(1)\right)v(1)\Big]dt=\rho\int_{0}^{T} (z_{x}, v)dt,	
\end{align}
provided we have
\[(w^{m_{n}}w_{x}^{m_{n}}, v)\rightarrow (ww_{x}, v),
\]
\[\sum_{i=0}^{1}(w^{m_{n}})^{2}(i,t) w_{t}^{m_{n}}(i,t)\rightarrow \sum_{i=0}^{1}w^{2}(i,t) w_{t}(i,t),
\]
and \[\sum_{i=0}^{1}(w^{m_{n}})^{3}(i,t) w^{m_{n}}(i,t)\rightarrow \sum_{i=0}^{1}w^{3}(i,t) w(i,t),
\] as \(n\rightarrow \infty, \) (see \cite{MR4019979} for more details).

For the proof of \(w(0)=w_{0}\) see \cite{MR4019979}, and for the proof of uniqueness see \cite{MR3790146}.

\section{Finite Element Approximation.}
\label{4}
This section contains the semi-discrete scheme of the continuous problem \eqref{1.13} and exponential bounds of the semi-discrete scheme. Moreover, we establish error estimates for the state variable and feedback control laws.

For any positive integer \(N\), suppose that \(p=\{0=x_{0}<x_{1}<\cdots< x_{N}=1\}\) is the partition of the interval \([0,1]\). Let \(I_{j}=(x_{j-1}, x_{j}), 1\leq j\leq N,\) be the sub-intervals of the partition \(p\) with mesh length  \(h_{j}=x_{j}-x_{j-1}\). Assume that mesh parameter $h=\max\limits_{1\leq j\leq N}h_{j}$.
We construct a finite dimensional subspace of \(H^{1}\) as
\[ S_{h}=\{v_{h}\in C^{0}([0,1]),\  v_{h}|_{I_{j}}\in P_{1}(I_{j}), 1\leq j\leq N\},
\]
where \(P_{1}\) is a polynomial of degree at most one on each \(I_{j}, 1\leq j\leq N.\)

Now, the semi-discrete finite element scheme of the continuous problem \eqref{1.13} is to find \(w_{h}\in S_{h}\) such that
\begin{align}
	\label{4.1}
	\nonumber(w_{ht},\phi) +\nu (w_{hx},\phi_{x})+w_{d}(w_{hx},\phi) +(w_{h}w_{hx},\phi)+ \left((c_{0}+w_{d})w_{h}(0,t)+  \frac{2}{9c_{0}}w_{h}^{3}(0,t)-\rho z_{h}(0)\right)\phi(0)\\ +\left((c_{1}+w_{d})w_{h}(1,t)+ \frac{2}{9c_{1}}w_{h}^{3}(1,t)+\rho z_{h}(1)\right)\phi(1)=\rho (z_{hx}, \phi), \hspace{5mm} \forall \  \phi \in S_{h},
\end{align}
and set \(z_{h}=\int_{0}^{t}e^{-\delta(t-s)}w_{hx}(s)ds\), we can write
\begin{align}
	\label{4.2}
	z_{ht}+\delta z_{h} = w_{hx},
\end{align}
with $ w_{h}(x,0)=w_{0h}(x)$ and $ z_{h}(x,0)=0,$ where \(w_{0h}\) is an approximation of \(w_{0}\) in \(S_{h}\).

Since the space \(S_{h}\) is finite dimensional, the system \eqref{4.1} represents a non-linear system of ordinary differential equations. Therefore, using Picard's theorem, there exists a solution \(w_{h}(t)\) for some \(t\in(0, t^{*})\), where \(0\leq t^{*}\leq T.\) Hence, for global existence $\forall t>0$, we apply the following lemma, which shows the global stabilization for the semi-discrete scheme \eqref{4.1}.
\begin{lemma}
		\label{L4.1}
		Under  assumptions \((A1)\)- \((A2)\), there exists a decay \( 0<\alpha \leq \frac{1}{2}\min\{\nu, \delta,  c_{0}+w_{d}, c_{1}+w_{d}\}\) such that
		\begin{align*}
			\norm{w_{h}}^{2}&+\rho\norm{z_{h}}^{2}+\beta e^{-2\alpha t} \int_{0}^{t}e^{2\alpha s}\left(\norm{w_{hx}(s)}^{2}+\rho\norm{z_{h}(s)}^{2}+w_{h}^{2}(0,s)+w_{h}^{2}(1,s) \right)ds \nonumber\\&+ \frac{e^{-2\alpha t}}{3c_{0}}\int_{0}^{t}e^{2\alpha s}w_{h}^{4}(0,s)ds +\frac{e^{-2\alpha t}}{3c_{1}}\int_{0}^{t}e^{2\alpha s}w_{h}^{4}(1,s)ds\leq e^{-2\alpha t} \norm{w_{0}}^{2},
		\end{align*}
		where $ 0<\beta= \min\{(\nu-2\alpha), (\delta-2\alpha), (c_{0}+w_{d}-2\alpha), (c_{1}+w_{d}-2\alpha)   \}.$  
	\end{lemma}
\begin{proof}
The proof is similar to Lemma \ref{L2.1}.
\end{proof}

\subsection{Error Estimates.}
This subsection deals with the error estimates for both the state variable and control inputs.

Here, we take the projection \cite{Yapping1990} in the following form:

\begin{align}
	\label{4.4}
	\nu\Big((w-\tilde{w}_{h})_{x}, \phi_{x}\Big)+\rho\Big(\int_{0}^{t}e^{-\delta(t-s)}\big(w(s)-\tilde{w}_{h}(s)\big)_{x}ds, \phi_{x}\Big)+\lambda\Big(w-\tilde{w}_{h}, \phi\Big)=0, \quad \forall \phi\in S_{h},
\end{align}
where \(\tilde{w}_{h}:[0, T]\rightarrow S_{h}\). This type of projection is known as the Ritz-Volterra projection. The existence and uniqueness of this projection has been shown in \cite{Yapping1990}. Set \(\eta = w-\tilde{w}_{h}\). The proof of the following lemma is given in \cite{ doi:10.1137/0727036, Yapping1990}.
\begin{lemma}
	\label{L4.2}
	Let \(w\in H^{2}(0,1)\) and \(w_{t}\in L^{2}(0,T; H^{2}(0,1)\cap H^{1}(0,1))\). Then the following holds
\[
	  \norm{\eta}_{j}\leq Ch^{2-j}\Big(\norm{w}_{2}+\int_{0}^{t}\norm{w(s)}_{2}\Big),\]
	  and 
	  \[\norm{\eta_{t}}_{j}\leq Ch^{2-j}\Big(\norm{w_{t}}_{2}+\int_{0}^{t}\norm{w_{t}(s)}_{2}\Big), 
	\]
where  \( j=0,1.\)
\end{lemma}
The following estimates of $\eta$ and $\eta_{t}$ at the boundary points $x=0, 1,$ are used in the proof of the error analysis. 
\begin{lemma}
	\label{L4.3}
	For \(x=0, 1\), the following estimates hold
	\[
		\norm{\eta(x)}\leq Ch^{2}\Big(\norm{w}_{2}+\int_{0}^{t}\norm{w(s)}_{2}ds \Big),  
	\]
and \[\norm{\eta_{t}(x)}\leq Ch^{2}\Big(\norm{w_{t}}_{2}+\int_{0}^{t}\norm{w_{t}(s)}_{2}ds \Big),\]
where \(C\) is a positive constant depends on \(\nu,\lambda,\) and \(\rho\).
\end{lemma}
\begin{proof}
	Let \(\psi\in H^{2}(0,1)\) be the solution of 
	\begin{align}
		\label{4.5}
		-\nu\psi_{xx}+\lambda \psi=0, \quad \psi_{x}(1)=1,\  \psi_{x}(0)=0,
	\end{align}
satisfying \(\norm{\psi}_{2}\leq C\), where \(C\) is a positive constant depends on \(\lambda\).

Multiplying by \(\eta\) in \eqref{4.5} and integrating with respect to \(x\) over \([0,1]\) yields
\begin{align*}
      \nu\eta(1,t)=\nu(\eta_{x}, \psi_{x})+\lambda(\eta, \psi).
\end{align*}
Using \eqref{4.4}, we can write
\begin{align}
	\label{4.6}
	\nu\eta(1,t)=\nu(\eta_{x}, \psi_{x}-\phi_{x})+\lambda(\eta, \psi-\phi)-\rho\big(\int_{0}^{t}e^{-\delta(t-s)}\eta_{x}(s)ds, \phi_{x}\big).
\end{align}
From the third term on the right hand side of \eqref{4.6}, we deduce that
\begin{align*}
	\big(\int_{0}^{t}\eta_{x}(s)ds, \phi_{x}\big)=-\big(\int_{0}^{t}\eta_{x}(s)ds, \psi_{x}-\phi_{x}\big)+\big(\int_{0}^{t}\eta_{x}(s)ds, \psi_{x}\big).
\end{align*}
Using integration by parts, we arrive at
\begin{align*}
	\big(\int_{0}^{t}\eta_{x}(s)ds, \phi_{x}\big)=-\big(\int_{0}^{t}\eta_{x}(s)ds, \psi_{x}-\phi_{x}\big)+\int_{0}^{t}\eta(1,s)ds-\big(\int_{0}^{t}\eta(s)ds, \psi_{xx}\big).
\end{align*}
Therefore, from \eqref{4.6}, we write
\begin{align*}
	\nu\eta(1,t)&=\nu(\eta_{x}, \psi_{x}-\phi_{x})+\lambda(\eta, \psi-\phi)+\rho\big(\int_{0}^{t}e^{-\delta(t-s)}\eta_{x}(s)ds, \psi_{x}-\phi_{x}\big)
	\\&\quad -\rho\int_{0}^{t}e^{-\delta(t-s)}\eta(1,s)ds+\rho\big(\int_{0}^{t}e^{-\delta(t-s)}\eta(s)ds, \psi_{xx}\big).
\end{align*}
Using H\"older's inequality, it follows that
\begin{align*}
	|\eta(1,t)|^{2}&\leq C\norm{\eta_{x}}^{2}\norm{\psi_{x}-\phi_{x}}^{2}+C\norm{\eta}^{2}\norm{\psi-\phi}^{2}+C\int_{0}^{t}\norm{\eta_{x}(s)}^{2}ds \norm{\psi_{x}-\phi_{x}}^{2}\\&\quad +\frac{\rho^{2}}{\delta}\int_{0}^{t}e^{-\delta(t-s)}|\eta(1,s)|^{2}ds+C\int_{0}^{t}\norm{\eta(s)}^{2}ds\norm{\psi_{xx}}^{2},
\end{align*}
where \(  \big(\int_{0}^{t}e^{-\delta(t-s)}\eta(1,s)ds\big)^{2}\leq \big(\int_{0}^{t}e^{-\delta(t-s)}ds\big)\big(\int_{0}^{t}e^{-\delta(t-s)}|\eta(1,s)|^{2}ds\big)\leq \frac{1}{\delta} \int_{0}^{t}e^{-\delta(t-s)}|\eta(1,s)|^{2}ds.\)

Using \cite[Chapter 1]{thomee2007galerkin} and Lemma \ref{L4.2}, we obtain
\begin{align*}
	|\eta(1,t)|^{2}&\leq Ch^{4}\norm{w}_{2}^{2}+Ch^{8}\norm{w}_{2}^{2}+Ch^{4}\int_{0}^{t}\norm{w(s)}_{2}^{2}ds  +\frac{\rho^{2}}{\delta}\int_{0}^{t}e^{-\delta(t-s)}|\eta(1,s)|^{2}ds.
\end{align*}
Using Gronwall's inequality, we have
\begin{align*}
	|\eta(1,t)|^{2}\leq Ch^{4}\Big(\norm{w}_{2}^{2}+\int_{0}^{t}\norm{w(s)}_{2}^{2}ds \Big)e^{\frac{\rho^{2}}{\delta}\int_{0}^{t}e^{-\delta(t-s)}ds}\leq Ch^{4}\Big(\norm{w}_{2}^{2}+\int_{0}^{t}\norm{w(s)}_{2}^{2}ds \Big).
\end{align*}
In a similar fashion, we can prove 
\begin{align*}
	|\eta_{t}(1,t)|^{2}\leq Ch^{4}\Big(\norm{w_{t}}_{2}^{2}+\int_{0}^{t}\norm{w_{t}(s)}_{2}^{2}ds \Big).
\end{align*}
A similar result holds for \(x=0.\)

This completes the proof.
\end{proof}
We define the error \(e:=w-w_{h}=:\eta-\theta\), where \(\eta=w-\tilde{w}_{h}\) and \(\theta=w_{h}-\tilde{w}_{h}\). 

Choose \(\tilde{w}_{h}(0)=w_{0h}\), so that \(\theta(0)=0.\) From Lemmas \ref{L4.2}-\ref{L4.3}, the estimates of \(\eta\) are known. Therefore, it is enough to estimate \(\theta.\)

Subtracting \eqref{4.1} from \eqref{1.13} and using Volterra projection \eqref{4.4}, we obtain
\begin{align}
	\label{4.7}
	(\theta_{t}, \phi)+\nu(\theta_{x}, \phi_{x})&+w_{d}(\theta_{x}, \phi)+(c_{0}+w_{d})\theta(0,t)\phi(0)-\rho \zeta(0,t)\phi(0)
	\nonumber\\&+(c_{1}+w_{d})\theta(1,t)\phi(1)+\rho \zeta(1,t)\phi(1)-\rho(\zeta_{x}, \phi)
	\nonumber\\&=(\eta_{t}, \phi)-\lambda(\eta, \phi)+w_{d}(\eta_{x}, \phi)+(c_{0}+w_{d})\eta(0,t)\phi(0)+(c_{1}+w_{d})\eta(1,t)\phi(1)
	\nonumber\\&\quad +(ww_{x}-w_{h}w_{hx}, \phi)+\frac{2}{9c_{0}}\Big(w^{3}(0,t)-w_{h}^{3}(0,t)\Big)\phi(0)+\frac{2}{9c_{1}}\Big(w^{3}(1,t)-w_{h}^{3}(1,t)\Big)\phi(1),
\end{align}
where \(\zeta:=\int_{0}^{t}e^{-\delta(t-s)}\theta_{x}(s)ds\)  and after differentiating with respect to time yields
\begin{align}
	\label{4.8}
	\zeta_{t}+\delta \zeta=\theta_{x}.
\end{align}
The cubic term \( w^{3}(i,t)-w_{h}^{3}(i,t)\) can be written as
\[w^{3}(i,t)-w_{h}^{3}(i,t)=\eta^{3}(i,t)-\theta^{3}(i,t)+3w(i,t)\eta(i,t)\Big(w(i,t)-\eta(i,t)\Big)-3w_{h}(i,t)\theta(i,t)\Big(w_{h}(i,t)-\theta(i,t)\Big),
\]
where \(i=0,1.\)

The following lemma deals with the proof of the error estimate for the state variable in the $L^{2}$-norm.
\begin{lemma}
	\label{L4.4}
    Suppose that  assumptions \((A1)\) and \((A2)\) hold. Then, there exists a positive constant $C$ independent of $h$ such that for \(0<\alpha\leq \frac{1}{2}\min\{\nu, \delta, c_{0}+\frac{w_{d}}{2}-\frac{\nu}{2}, c_{1}+w_{d}-\frac{\nu}{2} \}\) and \(\nu\leq 2\min\{c_{0}+\frac{w_{d}}{2}, c_{1}+w_{d}\}\) such that
   \begin{align*}
   	\norm{\theta}^{2}&+\rho \norm{\zeta}^{2}+\beta_{1}e^{-2\alpha t}\int_{0}^{t}e^{2\alpha s}\Big(\norm{\theta_{x}(s)}^{2}+ \rho\norm{\zeta(s)}^{2}+\theta^{2}(0,s)+\theta^{2}(1,s) \Big)ds
   	\nonumber\\&+e^{-2\alpha t}\int_{0}^{t}e^{2\alpha s}\Big(\frac{1}{9c_{0}}\theta^{4}(0,s)+\frac{1}{9c_{1}}\theta^{4}(1,s) \Big)ds\leq C(\norm{w_{0}}_{2}) h^{4} e^{-2\alpha t}\norm{w_{0}}_{3}^{2},
   \end{align*}
where \(0<\beta_{1}=\min\{(\nu-2\alpha), (c_{0}+\frac{w_{d}}{2}-\frac{\nu}{2}-2\alpha),(c_{1}+w_{d}-\frac{\nu}{2}-2\alpha),(\delta-2 \alpha)\}\).
\end{lemma}

\begin{proof}
    Choose \(\phi=\theta\) in \eqref{4.7} to obtain
    \begin{align}
    	\label{4.9}
    	\frac{1}{2}\frac{d}{dt}\norm{\theta}^{2}&+ \nu\norm{\theta_{x}}^{2}+(c_{0}+\frac{w_{d}}{2})\theta^{2}(0,t)+(c_{1}+\frac{3w_{d}}{2})\theta^{2}(1,t)+\frac{2}{9c_{0}}\theta^{4}(0,t)+\frac{2}{9c_{1}}\theta^{4}(1,t)
    	\nonumber\\&-\rho \zeta(0,t)\theta(0,t)+\rho \zeta(1,t)\theta(1,t)-(\zeta_{x}, \theta)
    	\nonumber\\&=(\eta_{t}-\lambda \eta, \theta)+w_{d}(\eta_{x}, \theta)+(c_{0}+w_{d})\eta(0,t)\theta(0,t)+(c_{1}+w_{d})\eta(1,t)\theta(1,t)
    	\nonumber\\&\quad +(ww_{x}-w_{h}w_{hx}, \theta)+\sum_{i=0}^{1}\Big(\frac{2}{9c_{i}}\eta^{3}(i,t)+\frac{2}{3c_{i}}\eta(i,t)w(i,t)\big(w(i,t)-\eta(i,t)\big)\Big)\theta(i,t)
    	\nonumber\\&\quad+\sum_{i=0}^{1}\Big(-\frac{2}{3c_{i}}w_{h}^{2}(i,t)\theta^{2}(i,t)+\frac{2}{3c_{i}}w_{h}(i,t)\theta^{3}(i,t)\Big).
    \end{align}
Multiplying \eqref{4.8} by \(\zeta\) and integrating with respect \(x\) over \([0,1]\), we get 
\begin{align}
	\label{4.10}
	\frac{1}{2}\frac{d}{dt}\norm{\zeta}^{2}+\delta\norm{\zeta}^{2}=\theta(1,t)\zeta(1,t)-\theta(0,t)\zeta(0,t)-(\zeta_{x}, \theta).
\end{align}
Again, multiplying \eqref{4.10} by \(\rho\) and adding in \eqref{4.9}, it follows that
\begin{align}
	\label{4.11}
	\frac{1}{2}\frac{d}{dt}\Big(\norm{\theta}^{2}&+\rho \norm{\zeta}^{2} \Big)+\nu\norm{\theta_{x}}^{2}+\rho \delta\norm{\zeta}^{2}+(c_{0}+\frac{w_{d}}{2})\theta^{2}(0,t)+(c_{1}+\frac{3w_{d}}{2})\theta^{2}(1,t)
	\nonumber\\&+\frac{2}{9c_{0}}\theta^{4}(0,t)+\frac{2}{9c_{1}}\theta^{4}(1,t)
	\nonumber\\&=(\eta_{t}-\lambda \eta, \theta)+w_{d}(\eta_{x}, \theta)+(c_{0}+w_{d})\eta(0,t)\theta(0,t)+(c_{1}+w_{d})\eta(1,t)\theta(1,t)
	\nonumber\\&\quad +(ww_{x}-w_{h}w_{hx}, \theta)+\sum_{i=0}^{1}\Big(\frac{2}{9c_{i}}\eta^{3}(i,t)+\frac{2}{3c_{i}}\eta(i,t)w(i,t)\big(w(i,t)-\eta(i,t)\big)\Big)\theta(i,t)
	\nonumber\\&\quad+\sum_{i=0}^{1}\Big(-\frac{2}{3c_{i}}w_{h}^{2}(i,t)\theta^{2}(i,t)+\frac{2}{3c_{i}}w_{h}(i,t)\theta^{3}(i,t)\Big),
	\nonumber\\&=\sum_{j=1}^{6}I_{j}.
\end{align}
The first term on the right hand side of \eqref{4.11} is bounded by
\begin{align*}
	I_{1}=(\eta_{t}-\lambda \eta, \theta)\leq C\Big(\norm{\eta_{t}}^{2}+ \norm{\eta}^{2}\Big)+\frac{\nu}{8}\norm{\theta}^{2}.
\end{align*}
Using integration by parts, the second term \(I_{2}\) on the right hand side of \eqref{4.11} gives 
\begin{align*}
 I_{2}=w_{d}(\eta_{x}, \theta)&=w_{d}\eta(1,t)\theta(1,t)-w_{d}\eta(0,t)\theta(0,t)-w_{d}(\eta, \theta_{x})
 \\&\leq C\Big(\eta^{2}(0,t)+\eta^{2}(1,t)+\norm{\eta}^{2}\Big)+\frac{w_{d}}{4}\theta^{2}(0,t)+w_{d} \theta^{2}(1,t)+\frac{\nu}{12}\norm{\theta_{x}}^{2}.
\end{align*}
A use of the Young's inequality, the third term \(I_{3}\) gives
\begin{align*}
   I_{3}=(c_{0}+w_{d})\eta(0,t)\theta(0,t)+(c_{1}+w_{d})\eta(1,t)\theta(1,t),
   \leq C\Big(\eta^{2}(0,t)+\eta^{2}(1,t)\Big)+\frac{c_{0}}{8}\theta^{2}(0,t)+\frac{c_{1}}{8}\theta^{2}(1,t).
\end{align*}
Rewrite the fourth term \(I_{4}\) on the right hand side of \eqref{4.11} as 
\begin{align}
	\label{4.12}
	I_{4}=(ww_{x}-w_{h}w_{hx}, \theta)=\big(w_{x}(\eta-\theta), \theta\big)+(w_{h}\eta_{x}, \theta)-(w_{h}\theta_{x}, \theta).
\end{align}
First subterms term of $I_{4}$ on the right hand side of \eqref{4.12} is bounded by
\begin{align*}
	\big(w_{x}(\eta-\theta), \theta\big)&\leq \norm{w_{x}}_{\infty}\norm{\eta}\norm{\theta}+ \norm{w_{x}}_{\infty}\norm{\theta}\norm{\theta},
	\\&\leq C\norm{\eta}^{2}+ \frac{\nu}{8}\norm{\theta}^{2}+C\norm{w_{x}}_{\infty}^{2}\norm{\theta}^{2}.
\end{align*}
Setting \(w_{h}=\tilde{w}_{h}+\theta\) into second subterms of $I_{4}$ on the right hand side of \eqref{4.12} yields
\begin{align*}
	(w_{h}\eta_{x}, \theta)=\Big( (\tilde{w}_{h}+\theta)\eta_{x}, \theta\Big)&=\tilde{w}_{h}(1,t)\eta(1,t)\theta(1,t)-\tilde{w}_{h}(0,t)\eta(0,t)\theta(0,t)-(\eta, \tilde{w}_{hx}\theta+\tilde{w}_{h}\theta_{x} ),
	\\&\leq C\Big(\tilde{w}_{h}^{2}(1,t)\eta^{2}(1,t)+\tilde{w}_{h}^{2}(0,t)\eta^{2}(0,t)\Big)+\frac{c_{0}}{8}\theta^{2}(0,t)+\frac{c_{1}}{8}\theta^{2}(1,t)
	\\&\quad + C\Big(\norm{\tilde{w}_{hx}}_{\infty}^{2}+\norm{\tilde{w}_{h}}_{\infty}^{2}\Big)\norm{\theta}^{2}+\frac{\nu}{12}\norm{\theta_{x}^{2}}+C\norm{\eta}^{2}.
\end{align*}
Last subterms of $I_{4}$ on the right hand of \eqref{4.12} is bounded by 
\begin{align*}
	-(w_{h}\theta_{x}, \theta)\leq \norm{w_{h}}_{\infty}\norm{\theta_{x}}\norm{\theta}\leq C\norm{w_{h}}_{\infty}^{2}\norm{\theta}^{2}+ \frac{\nu}{12}\norm{\theta_{x}^{2}}.
\end{align*}
Now, from the fifth term \(I_{5}\) on the right hand side of \eqref{4.11}, we arrive at
\[
	\frac{2}{9c_{i}}\eta^{3}(i,t)\theta(i,t)\leq C\eta^{6}(i,t)+\frac{c_{i}}{8}\theta^{2}(i,t).
\]
 Using Young's inequality subterms of $I_{5}$ yields
\[ \frac{2}{3c_{i}}w(i,t)\eta(i,t)\big(w(i,t)-\eta(i,t)\big)\theta(i,t)\leq C\big(\eta^{2}(i,t)w^{4}(i,t)+ \eta^{4}(i,t)w^{2}(i,t)\big)+\frac{c_{i}}{8}\theta^{2}(i,t), 
\]
where $i=0,1.$

Finally, the final term on the right hand side of \eqref{4.11} is given by
\begin{align*}
	I_{6}=\frac{2}{3c_{i}}w_{h}(i,t)\theta^{3}(i,t)\leq \frac{6}{36c_{i}}\theta^{4}(i,t)+ \frac{2}{3c_{i}}\theta^{2}(i,t)w_{h}^{2}(i,t), \quad i=0,1.
\end{align*}
Substituting the bounds for \(I_{j}, j=1,2,\ldots,6\) into \eqref{4.11}, it follows that
\begin{align*}
\frac{1}{2}\frac{d}{dt}\Big(\norm{\theta}^{2}&+\rho \norm{\zeta}^{2} \Big)+\frac{3\nu}{4}\norm{\theta_{x}}^{2}+\rho \delta\norm{\zeta}^{2}+(\frac{c_{0}}{2}+\frac{w_{d}}{4})\theta^{2}(0,t)+(\frac{c_{1}}{2}+\frac{w_{d}}{2})\theta^{2}(1,t)
\nonumber\\&+\frac{1}{18c_{0}}\theta^{4}(0,t)+\frac{1}{18c_{1}}\theta^{4}(1,t)
\nonumber\\&\leq \frac{\nu}{4}\norm{\theta}^{2}+ C \Big(\norm{w_{x}}_{\infty}^{2}+\norm{\tilde{w}_{hx}}_{\infty}^{2}+\norm{\tilde{w}_{h}}_{\infty}^{2}+\norm{{w}_{h}}_{\infty}^{2} \Big)\norm{\theta}^{2}
\nonumber\\&\quad + C\bigg(\norm{\eta_{t}}^{2}+\norm{\eta}^{2}+\sum_{i=0}^{1}\Big(\eta^{2}(i,t)+\eta^{2}(i,t)w^{4}(i,t)+\eta^{4}(i,t)w^{2}(i,t)+\eta^{6}(i,t)\Big)\bigg).
\end{align*}
Using \(\norm{\tilde{w}_{h}}_{\infty}\leq C\norm{w}_{1}\) and Poincar\'e-Wirtinger's inequality, we arrive at

\begin{align}
	\label{4.13}
	\frac{1}{2}\frac{d}{dt}\Big(\norm{\theta}^{2}&+\rho \norm{\zeta}^{2} \Big)+\frac{\nu}{2}\norm{\theta_{x}}^{2}+\rho \delta\norm{\zeta}^{2}+(\frac{c_{0}}{2}+\frac{w_{d}}{4}-\frac{\nu}{4})\theta^{2}(0,t)+(\frac{c_{1}}{2}+\frac{w_{d}}{2}-\frac{\nu}{4})\theta^{2}(1,t)
	\nonumber\\&+\frac{1}{18c_{0}}\theta^{4}(0,t)+\frac{1}{18c_{1}}\theta^{4}(1,t)
	\nonumber\\&\leq  C \Big(\norm{w}_{2}^{2}+\norm{w}_{1}^{2}+2\norm{|{w}_{h}|}^{2} \Big)\norm{\theta}^{2}
	\nonumber\\&\quad + C\bigg(\norm{\eta_{t}}^{2}+\norm{\eta}^{2}+\sum_{i=0}^{1}\Big(\eta^{2}(i,t)+\eta^{2}(i,t)w^{4}(i,t)+\eta^{4}(i,t)w^{2}(i,t)+\eta^{6}(i,t)\Big)\bigg).
\end{align}
Multiplying \eqref{4.13} by \(2e^{2\alpha t}\) and using Poincar\'e-Wirtinger's inequality, we have
\begin{align*}
	\frac{d}{dt}\Big(\norm{e^{\alpha t}\theta}^{2}&+\rho \norm{e^{\alpha t}\zeta}^{2} \Big)+(\nu-2\alpha)e^{2\alpha t}\norm{\theta_{x}}^{2}
	+(\rho\delta-2\rho \alpha)e^{2\alpha t}\norm{\zeta}^{2}+(c_{0}+\frac{w_{d}}{2}-\frac{\nu}{2}-2\alpha)e^{2\alpha t}\theta^{2}(0,t)
	\nonumber\\&+(c_{1}+w_{d}-\frac{\nu}{2}-2\alpha)e^{2\alpha t}\theta^{2}(1,t)
	+\frac{e^{2\alpha t}}{9c_{0}}\theta^{4}(0,t)+\frac{e^{2\alpha t}}{9c_{1}}\theta^{4}(1,t)
	\nonumber\\&\leq Ce^{2\alpha t} \Big(\norm{w}_{2}^{2}+\norm{w}_{1}^{2}+2\norm{|{w}_{h}|}^{2} \Big)\norm{\theta}^{2}
	\nonumber\\&\quad + Ce^{2\alpha t}\bigg(\norm{\eta_{t}}^{2}+\norm{\eta}^{2}+\sum_{i=0}^{1}\Big(\eta^{2}(i,t)+\eta^{2}(i,t)w^{4}(i,t)+\eta^{4}(i,t)w^{2}(i,t)+\eta^{6}(i,t)\Big)\bigg).
\end{align*}
Integrating with respect to time over \([0,t]\), we observe that
\begin{align*}
	\norm{e^{\alpha t}\theta}^{2}&+\rho \norm{e^{\alpha t}\zeta}^{2}+\beta_{1}\int_{0}^{t}e^{2\alpha s}\Big(\norm{\theta_{x}(s)}^{2}+ \rho\norm{\zeta(s)}^{2}+\theta^{2}(0,s)+\theta^{2}(1,s) \Big)ds
\nonumber\\&+\int_{0}^{t}e^{2\alpha s}\Big(\frac{1}{9c_{0}}\theta^{4}(0,s)+\frac{1}{9c_{1}}\theta^{4}(1,s) \Big)ds
\nonumber\\&\leq C\int_{0}^{t}e^{2\alpha s} \Big(\norm{w(s)}_{2}^{2}+\norm{w(s)}_{1}^{2}+2\norm{|{w}_{h}(s)|}^{2} \Big)\norm{\theta(s)}^{2}ds
\nonumber\\& \quad + C\int_{0}^{t}e^{2\alpha s}\bigg(\sum_{i=0}^{1}\Big(\eta^{2}(i,s)+\eta^{2}(i,s)w^{4}(i,s)+\eta^{4}(i,s)w^{2}(i,s)+\eta^{6}(i,s)\Big)
\nonumber\\& \quad \quad +\norm{\eta_{t}(s)}^{2}+\norm{\eta(s)}^{2}\bigg)ds,
\end{align*}
where \(0<\beta_{1}=\min\{(\nu-2\alpha), (c_{0}+\frac{w_{d}}{2}-\frac{\nu}{2}-2\alpha),(c_{1}+w_{d}-\frac{\nu}{2}-2\alpha),(\delta-2 \alpha)\}\).

Using the Gronwall's inequality,  Theorem \ref{thm2} with $\alpha=0$, and Lemmas \ref{L4.1}$(\alpha=0) $ and \ref{L4.2}-\ref{L4.3}, the proof is completed after multiplying by \(e^{-2\alpha t}\) in the resulting inequality.
\end{proof}
In the next lemma, we prove the error estimate for the state variable in the \(H^{1}\)-norm.
\begin{lemma}
	\label{L4.5}
	Let the assumptions \((A1)\) and \((A2)\) hold. Then the following holds:
	\begin{align*}
		\nu \norm{\theta_{x}}^{2}+ \sum_{i=0}^{1}\Big((c_{i}+w_{d})\theta^{2}(i,t)+\frac{2}{27c_{i}}\theta^{4}(i,t)\Big)+2\int_{0}^{t}e^{2\alpha s}\norm{\theta_{t}(s)}^{2}ds \leq h^{4}e^{-2\alpha t}C(\norm{w_{0}}_{3}).
	\end{align*}
\end{lemma}
\begin{proof}
	Choose \(\phi=\theta_{t}\) in \eqref{4.7} to obtain
\begin{align}
	\label{4.14}
	\frac{1}{2}\frac{d}{dt}\Big(\nu \norm{\theta_{x}}^{2}&+ \sum_{i=0}^{1}\Big((c_{i}+w_{d})\theta^{2}(i,t)+\frac{1}{9c_{i}}\theta^{4}(i,t)\Big)\Big)+\norm{\theta_{t}}^{2}
	\nonumber\\&=-\rho(\zeta, \theta_{xt})+(\eta_{t}-\lambda \eta, \theta_{t})+w_{d}(\eta_{x}-\theta_{x}, \theta_{t})+(ww_{x}-w_{h}w_{hx}, \theta_{t})
	\nonumber\\&\quad +\sum_{i=0}^{1}\Big((c_{i}+w_{d})\eta(i,t)\theta_{t}(i,t)+\frac{2}{9c_{i}}\eta^{3}(i,t)\theta_{t}(i,t)\Big)
	\nonumber\\&\quad +\sum_{i=0}^{1}\frac{2}{3c_{i}}\Big(w(i,t)\eta(i,t)\big(w(i,t)-\eta(i,t)\big)\theta_{t}(i,t)\Big)
	\nonumber\\&\quad -\sum_{i=0}^{1}\frac{2}{3c_{i}}\Big(w_{h}(i,t)\theta(i,t)\big(w_{h}(i,t)-\theta(i,t)\big)\theta_{t}(i,t)\Big).
\end{align}
Also, from \eqref{4.8}, we can write
\begin{align}
	\label{4.15}
	(\zeta, \theta_{xt})=\delta(\zeta, \theta_{x})-\norm{\theta_{x}}^{2}+\frac{d}{dt}(\zeta, \theta_{x}).
\end{align}
From \eqref{4.14} and \eqref{4.15}, we arrive at 
\begin{align}
	\label{4.16}
	\frac{1}{2}\frac{d}{dt}\Big(\nu \norm{\theta_{x}}^{2}&+ \sum_{i=0}^{1}\Big((c_{i}+w_{d})\theta^{2}(i,t)+\frac{1}{9c_{i}}\theta^{4}(i,t)\Big)\Big)
	\nonumber\\&=-\rho\frac{d}{dt}(\zeta, \theta_{x})+\rho \norm{\theta_{x}}^{2}-\rho \delta(\zeta, \theta_{x}) +(\eta_{t}-\lambda \eta, \theta_{t})+w_{d}(\eta_{x}-\theta_{x}, \theta_{t})
	\nonumber\\&\quad +(ww_{x}-w_{h}w_{hx}, \theta_{t})
	\nonumber\\&\quad +\sum_{i=0}^{1}\Big((c_{i}+w_{d})\eta(i,t)\theta_{t}(i,t)+\frac{2}{9c_{i}}\eta^{3}(i,t)\theta_{t}(i,t)\Big)
	\nonumber\\&\quad +\sum_{i=0}^{1}\frac{2}{3c_{i}}\Big(w(i,t)\eta(i,t)\big(w(i,t)-\eta(i,t)\big)\theta(i,t)\Big)
	\nonumber\\&\quad -\sum_{i=0}^{1}\frac{2}{3c_{i}}\Big(w_{h}(i,t)\theta(i,t)\big(w_{h}(i,t)-\theta(i,t)\big)\theta_{t}(i,t)\Big).
\end{align}
The fifth term on the right hand side of \eqref{4.16} is given by
\begin{align*}
	w_{d}(\eta_{x}-\theta_{x}, \theta_{t})&=w_{d}\frac{d}{dt}\Big(\eta(1,t)\theta(1,t)-\eta(0,t)\theta(0,t)\Big)-w_{d}\Big(\eta_{t}(1,t)\theta(1,t)-\eta_{t}(0,t)\theta(0,t)\Big)
	\\&\quad -w_{d}\frac{d}{dt}(\eta, \theta_{x})+w_{d}(\eta_{t}, \theta_{x})-w_{d}(\theta_{x}, \theta_{t}),
	\\&\leq w_{d}\frac{d}{dt}\Big(\eta(1,t)\theta(1,t)-\eta(0,t)\theta(0,t)\Big)-w_{d}\frac{d}{dt}(\eta, \theta_{x})+C\Big(\eta_{t}^{2}(1,t)+\eta_{t}^{2}(0,t)+\norm{\eta_{t}}^{2}\Big)
	\\&\quad + C\norm{\theta_{x}}^{2}+\epsilon\norm{\theta_{t}}^{2}+\theta^{2}(0,t)+\theta^{2}(1,t),
\end{align*}
where \(\epsilon>0\) will be chosen later.

The sixth term on the right hand side of \eqref{4.16} can be written as
\begin{align}
	\label{4.17}
	(ww_{x}-w_{h}w_{hx}, \theta_{t})=\Big(w_{x}\big(\eta-\theta\big), \theta_{t}\Big)+(w_{h}\eta_{x}, \theta_{t})-(w_{h}\theta_{x}, \theta_{t}).
\end{align}
The first term on the right hand side of \eqref{4.17} is bounded by 
\begin{align*}
	\Big(w_{x}\big(\eta-\theta\big), \theta_{t}\Big)\leq C\norm{w_{x}}_{\infty}^{2}\norm{\eta}^{2}+C\norm{w_{x}}_{\infty}^{2}\norm{\theta}^{2}+\epsilon\norm{\theta_{t}}^{2}.
\end{align*}
Setting \(w_{h}=\theta+\tilde{w}_{h}\), the second term on the right hand side of \eqref{4.17} yields
\begin{align*}
	(w_{h}\eta_{x}, \theta_{t})&=(\theta\eta_{x}, \theta_{t})+(\tilde{w}_{h}\eta_{x}, \theta_{t}),
	\\&=(\theta\eta_{x}, \theta_{t})+\tilde{w}_{h}(1,t)\eta(1,t)\theta_{t}(1,t)-\tilde{w}_{h}(0,t)\eta(0,t)\theta_{t}(0,t)-(\tilde{w}_{hx}\eta, \theta_{t})-(\tilde{w}_{h}\eta, \theta_{xt}),
	\\&=(\theta\eta_{x}, \theta_{t})+\frac{d}{dt}\Big(\tilde{w}_{h}(1,t)\eta(1,t)\theta(1,t)\Big)-\big(\tilde{w}_{h}(1,t)\eta(1,t)\big)_{t}\theta(1,t)-\frac{d}{dt}\Big(\tilde{w}_{h}(0,t)\eta(0,t)\theta(0,t)\Big)
	\\&\quad +\big(\tilde{w}_{h}(0,t)\eta(0,t)\big)_{t}\theta(0,t)-(\tilde{w}_{hx}\eta, \theta_{t})-\frac{d}{dt}(\tilde{w}_{h}\eta, \theta_{x})+((\tilde{w}_{h}\eta)_{t}, \theta_{x}),
	\\&\leq C\norm{|\theta|}^{2}\norm{\eta_{x}}^{2}+\epsilon\norm{\theta_{t}}^{2}+C\norm{\theta_{x}}^{2}+C\Big(\norm{\tilde{w}_{hx}}_{\infty}^{2}\norm{\eta}^{2}+\norm{\tilde{w}_{ht}}_{\infty}^{2}\norm{\eta_{t}}^{2}\Big)
	\\&\quad +C \norm{\tilde{w}_{ht}}_{\infty}^{2}\big(\eta_{t}^{2}(1,t)+\eta_{t}^{2}(0,t)\big)+\theta^{2}(1,t)+\theta^{2}(0,t)-\frac{d}{dt}(\tilde{w}_{h}\eta, \theta_{x})
	\\&\quad +\frac{d}{dt}\Big(\tilde{w}_{h}(1,t)\eta(1,t)\theta(1,t)\Big)-\frac{d}{dt}\Big(\tilde{w}_{h}(0,t)\eta(0,t)\theta(0,t)\Big).
\end{align*}
The third term on the right hand side of \eqref{4.17} is estimated by 
\begin{align*}
	-(w_{h}\theta_{x}, \theta_{t})\leq C\norm{|w_{h}|}^{2}\norm{\theta_{x}}^{2}+ \epsilon\norm{\theta_{t}}^{2}.
\end{align*}
Now, the seventh term on the right hand side of \eqref{4.16} is deduced by 
\begin{align*}
	(c_{i}+w_{d})\eta(i,t)\theta_{t}(i,t)+\frac{2}{9c_{i}}\eta^{3}(i,t)\theta_{t}(i,t)&\leq (c_{i}+w_{d})\frac{d}{dt}\Big(\eta(i,t)\theta(i,t)\Big)+\frac{2}{9c_{i}}\frac{d}{dt}\Big(\eta^{3}(i,t)\theta(i,t)\Big)
	\\&\quad+C\eta_{t}^{2}(i,t)+\theta^{2}(i,t)+C\eta_{t}^{2}(i,t)\eta^{2}(i,t)+C\eta^{2}(i,t)\theta^{2}(i,t),
\end{align*}
where \(i=0,1.\)

On the right hand side of \eqref{4.16}, the eighth term is given by
\begin{align*}
	\frac{2}{3c_{i}}\Big(w(i,t)\eta(i,t)\big(w(i,t)-\eta(i,t)\big)\theta_{t}(i,t)\Big)&\leq \frac{2}{3c_{i}}\frac{d}{dt}\Big(w^{2}(i,t)\eta(i,t)\theta(i,t)\Big)-\frac{2}{3c_{i}}\frac{d}{dt}\Big(w(i,t)\eta^{2}(i,t)\theta(i,t)\Big)
	\\&\quad+C\Big(\eta^{2}(i,t)+\eta_{t}^{2}(i,t)+\theta^{2}(i,t)\Big),
\end{align*}
where \(i=0,1.\)

Setting \(w_{h}=\theta+\tilde{w}_{h}\), the last term on the right hand side of \eqref{4.16} yields
\begin{align*}
	\frac{2}{3c_{i}}\Big(w_{h}(i,t)\theta(i,t)\big(w_{h}(i,t)-\theta(i,t)\big)\theta_{t}(i,t)\Big)&=\frac{2}{3c_{i}}\tilde{w}_{h}^{2}(i,t)\theta(i,t)\theta_{t}(i,t)-\frac{2}{3c_{i}}\tilde{w}_{h}(i,t)\theta^{2}(i,t)\theta_{t}(i,t)
	\\&\leq -\frac{1}{3c_{i}}\frac{d}{dt}\Big(\tilde{w}_{h}^{2}(i,t)\theta^{2}(i,t) \Big)-\frac{2}{9c_{i}}\frac{d}{dt}\Big(\tilde{w}_{h}(i,t)\theta^{3}(i,t) \Big)
	\\&\quad +C\Big(\norm{\tilde{w}_{h}}_{\infty}^{2}+\norm{\tilde{w}_{ht}}_{\infty}^{2} \Big)\theta^{2}(i,t)+C\theta^{4}(i,t),
\end{align*}
where \(i=0,1.\) 

Substituting above these estimates in \eqref{4.16} with \(\epsilon=\frac{1}{10}\), it follows from Theorem \ref{thm2} with \(\norm{\tilde{w}_{h}}_{\infty}\leq C\norm{w}_{1}\) and \(\norm{\tilde{w}_{hx}}_{\infty}\leq C\norm{w}_{2}\)
\begin{align*}
	\frac{1}{2}\frac{d}{dt}\Big(\nu \norm{\theta_{x}}^{2}&+ \sum_{i=0}^{1}\Big((c_{i}+w_{d})\theta^{2}(i,t)+\frac{1}{9c_{i}}\theta^{4}(i,t)\Big)\Big)+\frac{1}{2}\norm{\theta_{t}}^{2}
	\nonumber\\&\leq C\Big(\norm{|\theta|}^{2}\norm{\eta_{x}}^{2}+\norm{\eta}^{2}+\norm{\eta_{t}}^{2}+\norm{\theta_{x}}^{2}+\sum_{i=0}^{1}\big(\eta^{2}(i,t)+\eta_{t}^{2}(i,t)+\eta_{t}^{2}(i,t)\eta^{2}(i,t)\big) \Big)
	\nonumber\\&\quad +C \sum_{i=0}^{1}\Big(\eta^{2}(i,t)\theta^{2}(i,t)+\theta^{2}(i,t)+\theta^{4}(i,t) \Big)+ C\norm{|w_{h}|}^{2}\norm{\theta_{x}}^{2}+C\norm{\zeta}^{2}
	\nonumber\\&\quad + w_{d}\frac{d}{dt}\Big(\eta(1,t)\theta(1,t)-\eta(0,t)\theta(0,t)\Big)-w_{d}\frac{d}{dt}(\eta, \theta_{x})-\frac{d}{dt}(\tilde{w}_{h}\eta, \theta_{x})
	\nonumber\\&\quad +\frac{d}{dt}\Big(\tilde{w}_{h}(1,t)\eta(1,t)\theta(1,t)\Big)-\frac{d}{dt}\Big(\tilde{w}_{h}(0,t)\eta(0,t)\theta(0,t)\Big)
	\nonumber\\&\quad+\sum_{i=0}^{1}\bigg((c_{i}+w_{d})\frac{d}{dt}\Big(\eta(i,t)\theta(i,t)\Big)+\frac{2}{9c_{i}}\frac{d}{dt}\Big(\eta^{3}(i,t)\theta(i,t)\Big)\bigg)
	\nonumber\\&\quad+\sum_{i=0}^{1}\bigg( \frac{2}{3c_{i}}\frac{d}{dt}\Big(w^{2}(i,t)\eta(i,t)\theta(i,t)\Big)-\frac{2}{3c_{i}}\frac{d}{dt}\Big(w(i,t)\eta^{2}(i,t)\theta(i,t)\Big)\bigg)
	\nonumber\\&\quad-\sum_{i=0}^{1}\bigg(\frac{1}{3c_{i}}\frac{d}{dt}\Big(\tilde{w}_{h}^{2}(i,t)\theta^{2}(i,t) \Big)+\frac{2}{9c_{i}}\frac{d}{dt}\Big(\tilde{w}_{h}(i,t)\theta^{3}(i,t) \Big)\bigg)-\rho\frac{d}{dt}(\zeta, \theta_{x}).	
\end{align*}
Multiplying by $e^{2\alpha t}$ and integrating with respect to time over \([0,t]\), we get from Lemmas \ref{L2.1}-\ref{L2.2}
\begin{align*}
\nu e^{2\alpha t}\norm{\theta_{x}}^{2}&+ \sum_{i=0}^{1}e^{2\alpha t}\Big((c_{i}+w_{d})\theta^{2}(i,t)+\frac{1}{9c_{i}}\theta^{4}(i,t)\Big)+\int_{0}^{t}e^{2\alpha s}\norm{\theta_{t}(s)}^{2}ds \nonumber\\&\leq C\int_{0}^{t}e^{2\alpha s}\bigg(\norm{|\theta(s)|}^{2}\norm{\eta_{x}(s)}^{2}+\norm{\eta(s)}^{2}+\norm{\eta_{t}(s)}^{2}+\norm{\theta_{x}(s)}^{2}
\nonumber\\&\quad+\sum_{i=0}^{1}\big(\eta^{2}(i,s)+\eta_{t}^{2}(i,s)+\eta_{t}^{2}(i,s)\eta^{2}(i,s)\big) \bigg)ds+2w_{d}\Big(\eta(1,t)\theta(1,t)-\eta(0,t)\theta(0,t)\Big)
\nonumber\\&\quad +C(\alpha) \int_{0}^{t}e^{2\alpha s}\Big(\nu \norm{\theta_{x}(s)}^{2}+ \sum_{i=0}^{1}\Big(\theta^{2}(i,s)+\theta^{4}(i,s)+\eta^{4}(i,s)+\eta^{6}(i,s)\Big)\Big)ds
\nonumber\\&\quad -2w_{d}e^{2\alpha t}(\eta, \theta_{x})-2e^{2\alpha t}(\tilde{w}_{h}\eta, \theta_{x})
 +2e^{2\alpha t}\Big(\tilde{w}_{h}(1,t)\eta(1,t)\theta(1,t)\Big)-2e^{2\alpha t}\Big(\tilde{w}_{h}(0,t)\eta(0,t)\theta(0,t)\Big)
\nonumber\\&\quad+2e^{2\alpha t}\sum_{i=0}^{1}\bigg((c_{i}+w_{d})\Big(\eta(i,t)\theta(i,t)\Big)+\frac{2}{9c_{i}}\Big(\eta^{3}(i,t)\theta(i,t)\Big)\bigg)
\nonumber\\&\quad+2e^{2\alpha t}\sum_{i=0}^{1}\bigg( \frac{2}{3c_{i}}\Big(w^{2}(i,t)\eta(i,t)\theta(i,t)\Big)-\frac{2}{3c_{i}}\Big(w(i,t)\eta^{2}(i,t)\theta(i,t)\Big)\bigg)
\nonumber\\&\quad-2e^{2\alpha t}\sum_{i=0}^{1}\bigg(\frac{1}{3c_{i}}\Big(\tilde{w}_{h}^{2}(i,t)\theta^{2}(i,t) \Big)+\frac{2}{9c_{i}}\Big(\tilde{w}_{h}(i,t)\theta^{3}(i,t) \Big)\bigg)
\nonumber\\&\quad+C\int_{0}^{t}e^{2\alpha s}\Big(\norm{|w_{h}(s)|}^{2}\norm{\theta_{x}(s)}^{2}+\norm{\zeta(s)}^{2}\Big)ds-\rho e^{2\alpha t}(\zeta, \theta_{x}).
\end{align*}
Again, using the Cauchy-Schwarz inequality and Young's inequality, we obtain from Lemmas \ref{L2.1}-\ref{L2.2} with $\alpha=0$
\begin{align*}
	\frac{\nu}{2}e^{2\alpha t} \norm{\theta_{x}}^{2}&+e^{2\alpha t} \sum_{i=0}^{1}\Big(\frac{1}{2}(c_{i}+w_{d})\theta^{2}(i,t)+\frac{1}{27c_{i}}\theta^{4}(i,t)\Big)+\int_{0}^{t}e^{2\alpha s}\norm{\theta_{t}(s)}^{2}ds \nonumber\\&\leq C\int_{0}^{t}e^{2\alpha s}\bigg(\norm{|\theta(s)|}^{2}\norm{\eta_{x}(s)}^{2}+\norm{\eta(s)}^{2}+\norm{\eta_{t}(s)}^{2}+\norm{\theta_{x}(s)}^{2}
	\nonumber\\&\quad+\sum_{i=0}^{1}\big(\eta^{2}(i,s)+\eta_{t}^{2}(i,s)+\eta_{t}^{2}(i,s)\eta^{2}(i,s)\big) \bigg)ds+Ce^{2\alpha t}\norm{\eta}^{2}+Ce^{2\alpha t}\norm{\zeta}^{2}
	\nonumber\\&\quad+Ce^{2\alpha t}\sum_{i=0}^{1}\Big(\eta^{2}(i,t)+\eta^{4}(i,t)+\eta^{6}(i,t)\Big)+C\int_{0}^{t}e^{2\alpha s}\Big(\norm{|w_{h}(s)|}^{2}\norm{\theta_{x}(s)}^{2}+\norm{\zeta(s)}^{2}\Big)ds.
\end{align*}
Using Gronwall's inequality with Theorem \ref{thm2} and Lemma \ref{L4.1} with $\alpha=0$, the proof is completed after multiplying by $e^{-2\alpha t}$ in the resulting inequality using Lemmas \ref{L4.2}-\ref{L4.4}.
\end{proof}
\begin{remark}
  From Lemma \ref{L4.5}, we can write using Poincar\'e inequality
  \begin{align*}
     \norm{\theta}_{\infty}\leq C(\norm{w_{0}}_{3}) h^{2} e^{C\norm{w_{0}}}.
  \end{align*}
\end{remark}
The following theorem shows the optimal order of convergence for the state variable.
\begin{theorem}
	\label{th4.1}
	Under the assumptions \((A1)\) and \((A2)\), there exists a constant \(C=C(\norm{w_{0}}_{3})\) independent of \(h\) such that 
	\begin{align*}
	\norm{w(t)-w_{h}(t)}_{j}\leq C h^{2-j} \exp(C\norm{w_{0}}_{2}), 
	\end{align*}
where \(j=0,1\) and \begin{align*}
	\norm{w(t)-w_{h}(t)}_{\infty}\leq C h^{2} \exp(C\norm{w_{0}}_{2}).
\end{align*}
\end{theorem}
\begin{proof}
	With the application of the triangle inequality, the proof follows from Lemmas \ref{L4.2}- \ref{L4.5}.
\end{proof}
Denotes \(e_{i}(t):=v_{i}(t)-v_{ih}(t)\), where \(v_{ih}(t):=\frac{1}{\nu}\bigg((c_{i} + w_{d})w_{h}(i,t) + \frac{2}{9c_{i}}w_{h}^3(i,t) - \rho z_{h}(i,t)\bigg), \) and \(z_{h}(i,t)=\int_{0}^{t}e^{-\delta(t-s)}w_{hx}(i,s)ds, i=0,1.\)

The following theorem presents the error estimate of the control inputs.
\begin{theorem}
	\label{th4.2}
Let assumptions \((A1)\) and \((A2)\) be true. Then, there exists a constant \(C\) independent of \(h\) such that
	\begin{align*}
		|e_{i}(t)|\leq  C(\norm{w_{0}}_{3}) h^{2},
	\end{align*}
where \(i=0,1.\)
\end{theorem}
\begin{proof}
Since
	\begin{align*}
		e_{i}(t)=\frac{1}{\nu}\bigg((c_{i} + w_{d})\big(w(i,t)-w_{h}(i,t)\big) + \frac{2}{9c_{i}}\big(w^{3}(i,t)-w_{h}^3(i,t)\big) - \rho \int_{0}^{t}e^{-\delta(t-s)}e_{i}(s)ds\bigg),
	\end{align*}
 we can write
\begin{align*}
	e_{i}(t)+\frac{\rho}{\nu}\int_{0}^{t}e^{-\delta(t-s)}e_{i}(s)ds&=\frac{1}{\nu}\bigg((c_{i} + w_{d})\big(w(i,t)-w_{h}(i,t)\big) + \frac{2}{9c_{i}}\big(w^{3}(i,t)-w_{h}^3(i,t)\big)\bigg),
	\\& =\frac{1}{\nu}\bigg((c_{i} + w_{d})\big(\eta(i,t)-\theta(i,t)\big)+\frac{2}{9c_{i}}\big(\eta^{3}(i,t)-\theta^{3}(i,t)+3\tilde{w}_{h}^{2}\eta(i,t)
	\\&\quad +3\tilde{w}_{h}\eta^{2}(i,t)-3\tilde{w}_{h}\theta^{2}(i,t)-3\tilde{w}_{h}^{2}\theta(i,t)\big)\bigg).
\end{align*}
With the help of the Cauchy-Schwarz inequality and Young's inequality, we obtain from Lemmas \ref{L2.1}-\ref{L2.2}
\begin{align*}
	|e_{i}(t)|\leq C\Big(|\eta(i,t)|+|\theta(i,t)|+|\eta(i,t)|^{2}+|\eta(i,t)|^{4}+|\theta(i,t)|^{2}+|\theta(i,t)|^{4}\Big)+\frac{\rho}{\nu}\int_{0}^{t}e^{-\delta(t-s)}|e_{i}(s)|ds.
\end{align*}
Using  Gronwall's inequality and Lemmas \ref{L4.2}-\ref{L4.3} and \ref{L4.5}, we arrive at
\begin{align*}
	|e_{i}(t)|\leq C(\norm{w_{0}}_{3}) h^{2} e^{\frac{\rho (1-e^{-\delta t})}{\nu \delta}}\leq C(\norm{w_{0}}_{3}) h^{2}.
\end{align*}
This completes the proof.
\end{proof}
\section{Numerical Examples.}
\label{5}
In this section, we present some numerical examples that illustrate our theoretical results. We examine the behavior of the state variable, control inputs, and the corresponding order of convergence.

Let  \(t_{n}=nk, \ n=0,1,\ldots, M\), where \(k\) is the time step size and \(M\) is a positive integer. Let \(W^{n}\) be the fully discrete approximation solution of the problem \eqref{1.13}  at \(t=t_{n}.\) Moreover, we denote by $V_{i}^{n}, i=0,1$ the fully discrete approximation of the feedback control input $v_{i}, i=0,1$ at \(t=t_{n}.\)  Suppose that \(\phi\) is a smooth function defined on \([0, \infty)\). We define $\phi(t_{n})=\phi^{n}$ and the difference operator \(\bar{\partial_{t}} {\phi^{n}}:=\frac{\phi^{n}-\phi^{n-1}}{k}\).

Applying the backward Euler method to \eqref{6.1}, we write a fully discrete scheme to seek a sequence $\{W^{n}\}_{n\geq 1}$ in the following form:
\begin{align}
	\label{6.2}
	\nonumber(\bar{\partial_{t}} {W}^{n},\phi) +\nu (W^{n}_{x},\phi_{x})+w_{d}(W^{n}_{x},\phi) +(W^{n}W^{n}_{x},\phi)+ \left((c_{0}+w_{d})W^{n}(0)+  \frac{2}{9c_{0}}(W^{n})^{3}(0)\right)\phi(0)\\ +\left((c_{1}+w_{d})W^{n}(1)+ \frac{2}{9c_{1}}(W^{n})^{3}(1)\right)\phi(1)+\rho (Z^{n}, \phi_{x})=0, 
\end{align}
where \(Z^{n}:=\int_{0}^{t_{n}}e^{-\delta(t_{n}-s)}W^{n}_{x}(s)ds\) and \(Z^{n}\) is approximated by right rectangular rule. We apply Newton's method to solve the nonlinear system \eqref{6.2}, where we use the previous time step \(W^{n}\) as an initial guess to find the current time step \(W^{n+1}\). 
 Since the exact solution of this problem is not known, we consider a refined mesh solution (reference solution)  as the exact solution, which is denoted by $w$.

In the following example, we analyze the behavior of the state variable and control inputs for various values of \(c_{0}\) and \(c_{1}\). Moreover, varying the values of \(c_{0}\), \(c_{1}\), and $\rho$, we investigate the order of convergence for the state variable in the \(L^{2}\) and \(L^{\infty}\)-norms and for the control inputs in the \(L^{\infty}\)-norm.
\begin{example}
	\label{ex1}
	We consider the initial condition \((t=0)\) \(w_{0}(x)=x(x-1)-w_{d}\) with \(w_{d}=3\) and \(x\in[0,1].\) We set the coefficients \(\nu=0.1\), \(\rho=1\), and \(\delta=1\). Further, we choose the spatial step size  $h=\frac{1}{100}$ and temporal step size $k=\frac{1}{100}$.
\end{example}
	 The system \eqref{eq:4.1}-\eqref{eq:4.8} in the time interval \([0, T]\), where \(T>0\), with zero Neumann boundary \((v_{0}(t),v_{1}(t)=0)\) is unstable in the \(L^{2}\)-norm, meaning that the approximate solution in the uncontrolled synthesis, which is represented by ``Uncontrolled solution" in Figure \ref{fig:ex1}(i) does not converge to its steady state solution. But in presence of the feedback control laws \eqref{v0} and \eqref{v1}, the state variable denoted as the `` Controlled solution"  in the same figure with respect to the \(L^{2}\)-norm goes to zero for various values of \(c_{0}\) and \(c_{1}\), which justifies the theoretical result in Lemma \ref{L2.1}. Therefore, the numerical simulation confirms that in the presence of feedback laws \eqref{v0} and \eqref{v1}, the solution of the closed loop system \eqref{eq:1.1}-\eqref{eq:1.4} goes to its steady state solution as time goes to infinity. For numerous values of \(c_{0}\) and \(c_{1}\), Figures \ref{fig:ex1}(ii) and (iii) depict the controller value for the left boundary \((x=0)\) and right boundary \((x=1)\), respectively. We observe that both controllers go to zero as time increases, which verifies the results in Lemma \ref{L2.2}.
	  
	    Next, we examine the order of convergence of the state variable and control inputs for numerous values of $c_{0}, c_{1},$ and $\rho$. We consider the reference solution with $h=\frac{1}{2048}$. From Table \ref{table:1}, we observe that the order of convergence of the state variable is two, and the error of the state variable decreases in both the \(L^{2}\) and \(L^{\infty}\)-norms for various values of \(h\) with the fixed values of \(k\) and \(c_{0}, c_{1}=0.1\). Furthermore, we obtain the same order of convergence for other values of \(c_{0}\) and \( c_{1}\), for example, \(c_{0}, c_{1}=10\). Hence, the order of convergence verifies the result in Theorem \ref{th4.1} for the state variable. Moreover, for varying the values of $\rho$, we get the same order of convergence for the state variable in the $L^{2}$ and $L^{\infty}$-norms. 
	  
	  Tables \ref{table:2} and \ref{table:3} contain the error and order of convergence with the spatial direction for both control inputs, and observe that the order of convergence is two for different values of \(c_{0}\) and \(c_{1}\). In addition, the error decreases in the \(L^{\infty}\)-norm. Therefore, Tables \ref{table:2} and \ref{table:3} justify the theoretical results of Theorem \ref{th4.2} for the control inputs. Moreover,  we notice that the order of convergence of the feedback controller at the left $(x=0)$ boundary and right $(x=1)$ boundary is two in the $L^{\infty}$-norm for various values of $\rho$.
	  
	  \begin{figure}[h]
	  	\centering
	  	(i)\includegraphics[width= 0.44\textwidth]{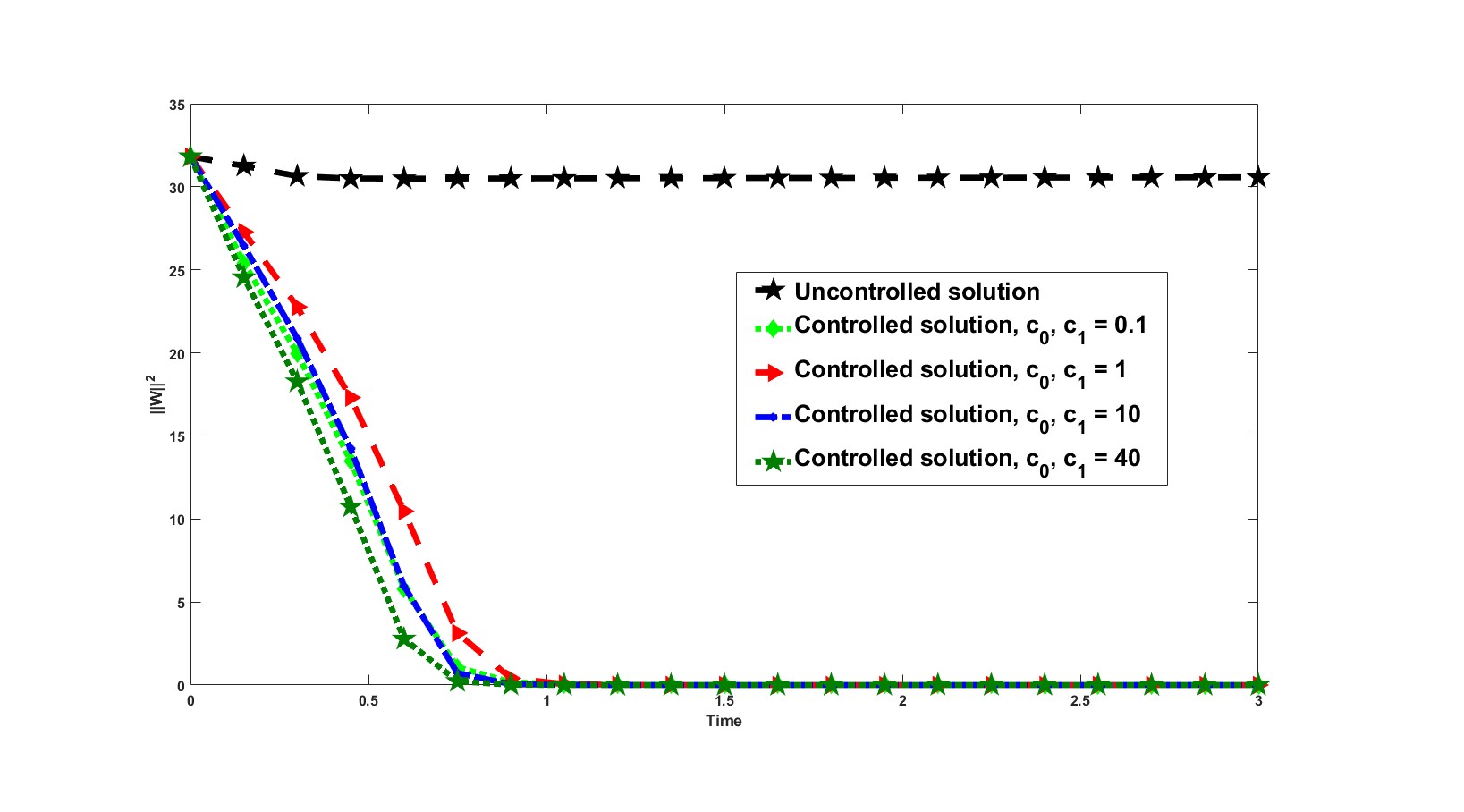}
	  	(ii)\includegraphics[width= 0.45\textwidth]{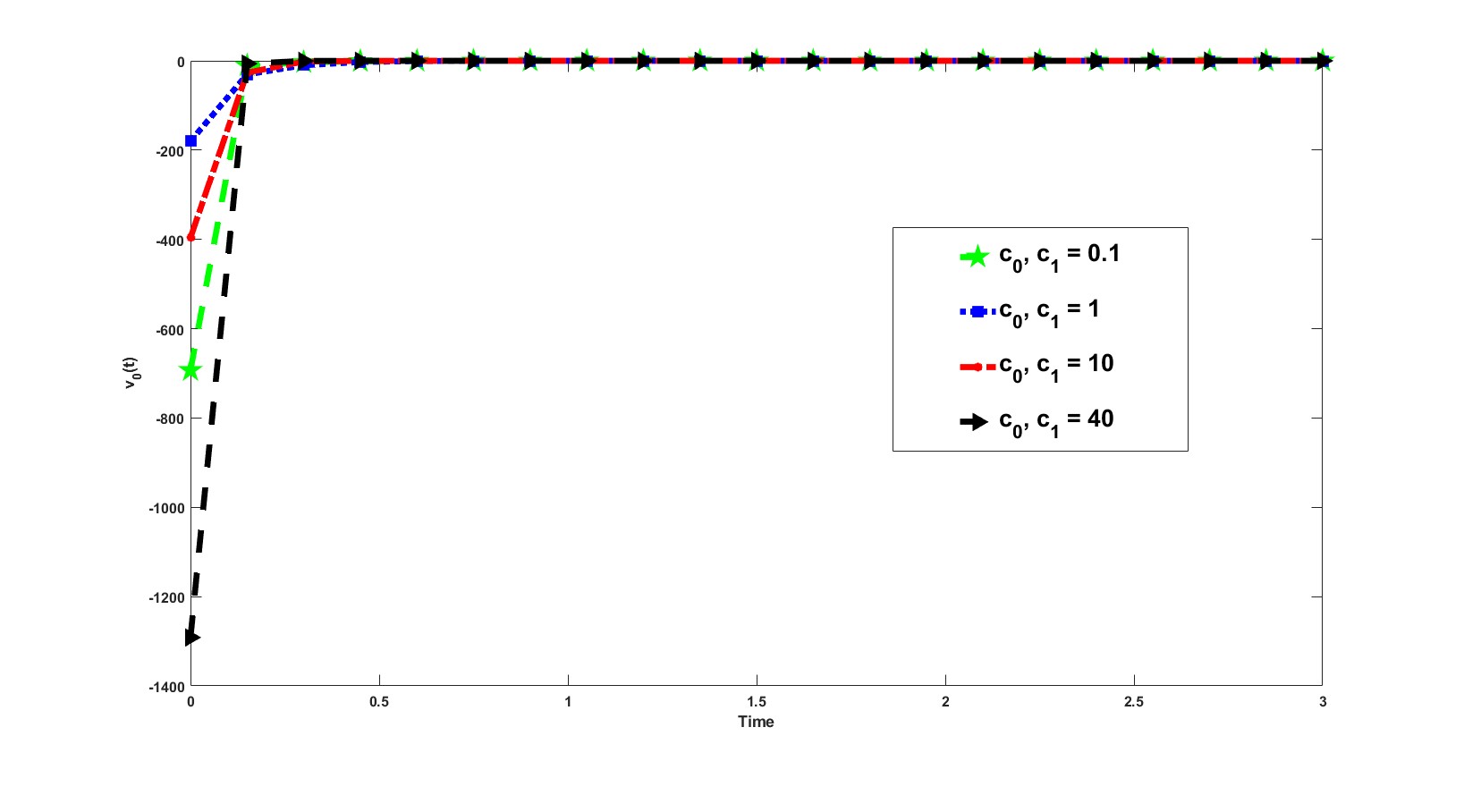}	
	  	(iii)\includegraphics[width= 0.40\textwidth]{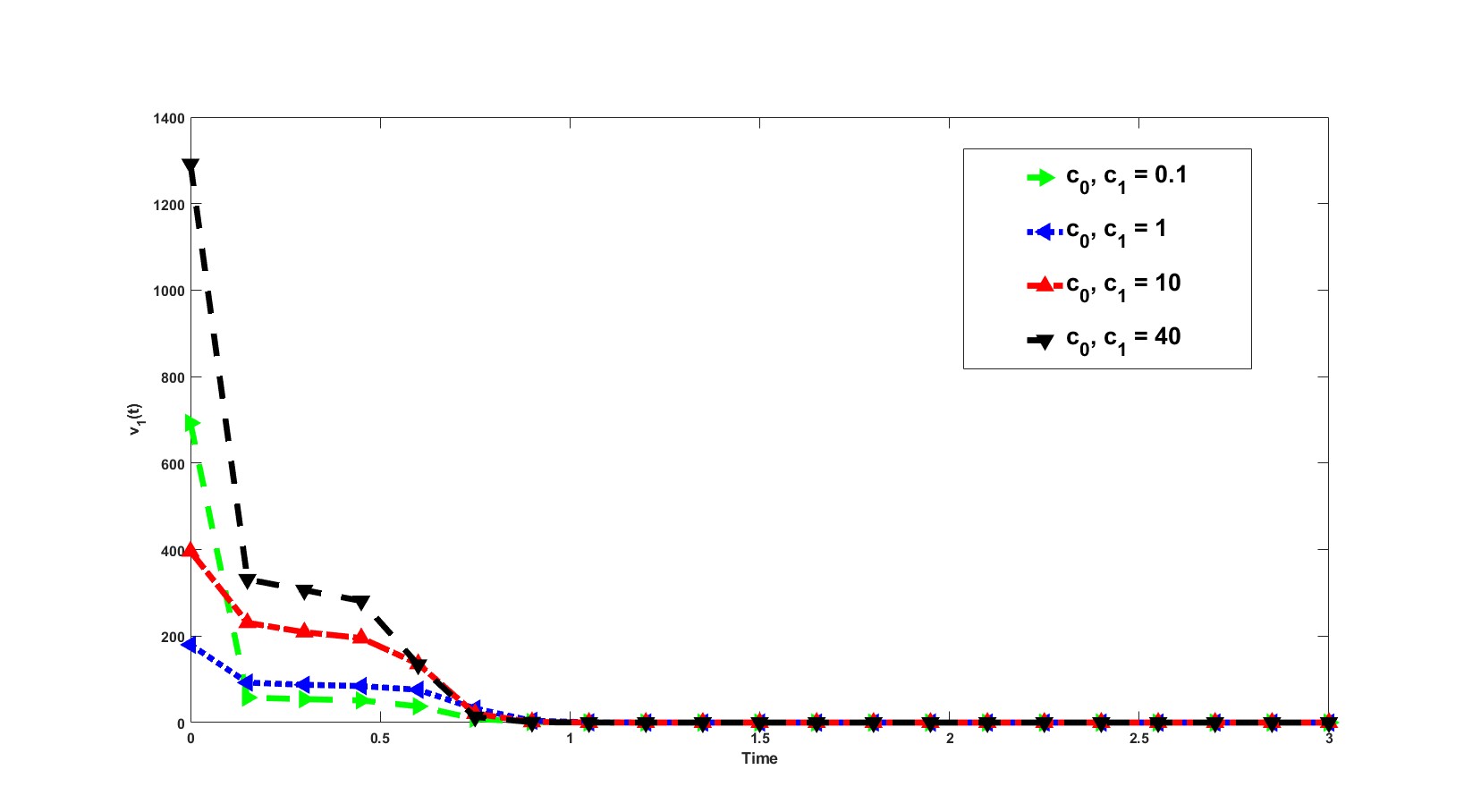}
	  	\caption{Example \ref{ex1}: {\bf (i)} Controlled and uncontrolled solution in the \(L^{2}\)-norm for  various values of \(c_{0}\) and \(c_{1}\).  
	  		{\bf (ii)} Control input for different values of \(c_{0}\) and \(c_{1}\) at the left boundary \(x=0\).  
	  		{\bf (iii)} Control input for various values of \(c_{0}\) and \(c_{1}\) at the right boundary \(x=1\).    
	  	}  
	  	\label{fig:ex1}
	  \end{figure}
\begin{table}[H]
	\centering
	\caption{ The order of convergence (O.C.) of the state variable in Example \ref{ex1}, considering varying the values of $h$ and a fixed value of $k=\frac{1}{1000}$.}
	\begin{tabular}{ c ||c|| c|| c|| c }
		\toprule
		h & $\norm{w-W}_{\infty}$         & O. C. & $\norm{w-W}$    &    O. C.    \\
		\midrule
		$\frac{1}{8}$& $ 4.8695e-05 $ &  $-- $ & $2.7887e-05$ &   $--$     \\
		
		$\frac{1}{16}$ &  $1.2323e-05 $ & $1.98$ & $6.7868e-06$&     $2.03$   \\
		
		$\frac{1}{32}$ &  $ 3.0955e-06  $ & $1.99$ & $1.663e-06$  &  $2.02$   \\
		
		$\frac{1}{64}$ & $7.7473e-07 $  & $ 1.99 $ &  $4.1061e-07$  &  $2.02$  \\
		
		$\frac{1}{128}$ &  $ 1.9377e-07$  & $1.99$ &  $1.0198e-07$  &  $2.00$  \\
		
	    $\frac{1}{256}$ &  $ 4.8477e-08$  & $2.00$ &  $2.5419e-08$  &  $2.00$  \\
		\bottomrule
	\end{tabular}
	\label{table:1}
\end{table}

\begin{table}[H]
	\centering
	\caption{ The order of convergence for the control inputs in Example \ref{ex1}, considering varying the values of $h$ and a fixed value of $k=\frac{1}{1000}$ with \(c_{0}, c_{1}=0.1\)}
	\begin{tabular}{ c ||c|| c|| c|| c }
		\toprule
		h & $\norm{v_{0}-V_{0}}_{\infty}$   & O. C. & $\norm{v_{1}-V_{1}}_{\infty}$    &  O. C.    \\
		\midrule
		$\frac{1}{8}$& $ 0.053$ &  $-- $ & $0.0180$ &   $--$     \\
		
		$\frac{1}{16}$ &  $0.0207 $ & $1.39$ & $0.0099$&     $0.86$   \\
		
		$\frac{1}{32}$ &  $ 0.0080 $ & $1.37$ & $0.0047$  &  $1.07$   \\
		
		$\frac{1}{64}$ & $0.0028 $  & $ 1.52 $ &  $0.0016$  &  $1.58$  \\
		
		$\frac{1}{128}$ &  $ 7.7891e-04$  & $1.85$ &  $4.4462e-04$  &  $1.85$  \\
		
		$\frac{1}{256}$ &  $ 1.9991e-04$  & $1.96$ &  $1.846e-04$  &  $1.93$  \\
		
		$\frac{1}{512}$ &  $4.9760e-05$  & $2.01$ &  $3.1751e-05$  &  $1.97$  \\
		\bottomrule
	\end{tabular}
	\label{table:2}
\end{table}

\begin{table}[H]
	\centering
	\caption{ The order of convergence for the control inputs in Example \ref{ex1}, considering varying values of $h$ and a fixed value of $k=\frac{1}{1000}$ with \(c_{0}, c_{1}=10\).}
	\begin{tabular}{ c ||c|| c|| c|| c }
		\toprule
		h & $\norm{v_{0}-V_{0}}_{\infty}$ & O. C. & $\norm{v_{1}-V_{1}}_{\infty}$    &   O. C.    \\
		\midrule
		$\frac{1}{8}$& $ 0.5905$ &  $-- $ & $0.2968$ &   $--$     \\
		
		$\frac{1}{16}$ &  $0.2897 $ & $1.03$ & $0.1720$&     $0.79$   \\
		
		$\frac{1}{32}$ &  $ 0.1275 $ & $1.18$ & $0.0803$  &  $1.09$   \\
		
		$\frac{1}{64}$ & $0.0436 $  & $ 1.55 $ &  $0.0265$  &  $1.59$  \\
		
		$\frac{1}{128}$ &  $ 0.0120$  & $1.86$ &  $0.0073$  &  $1.86$  \\
		
		$\frac{1}{256}$ &  $0.0031$  & $1.96$ &  $0.0019$  &  $1.95$  \\
		
		$\frac{1}{512}$ &  $7.6771e-04$  & $2.00$ &  $5.0238e-04$  &  $1.99$  \\
		\bottomrule
	\end{tabular}
	\label{table:3}
\end{table}


The next example shows the effect of the memory term for the state variable and control inputs. Moreover, we present the behavior of the state variable and control inputs for various values of the diffusion coefficient \(\nu\).
\begin{example}
	\label{ex2}
	We select the initial condition \(w_{0}(x)= 0.125 \cos(\pi x)-w_{d}\), with \(w_{d}=1\). We choose \(c_{0}, c_{1}=0.1\), \(\delta=5\), the spatial step size $h=\frac{1}{100}$, and temporal step size $k=\frac{1}{100}$. 
\end{example}
Figure \ref{fig:ex2}  shows the effect of the diffusion coefficient  while other parameters remain fixed. In Figure \ref{fig:ex2}(i), we observe that large values of \(\nu\) lead to faster decay for the state variable over time in the \(L^{2}\)-norm as here the decay rate is $\alpha= \min\{\nu, \delta, (c_{0}+w_{d}), (c_{1}+w_{d})\} = \nu$ ($\nu $ ranges from $0.1$ to $1$). Moreover, for various values of \(\nu \), Figures \ref{fig:ex2}(ii) and (iii) depict the controller values for the left \((x=0)\) boundary and right \((x=1)\) boundary, respectively. 
 \begin{figure}[ht]
	\centering
	(i)\includegraphics[width= 0.46\textwidth]{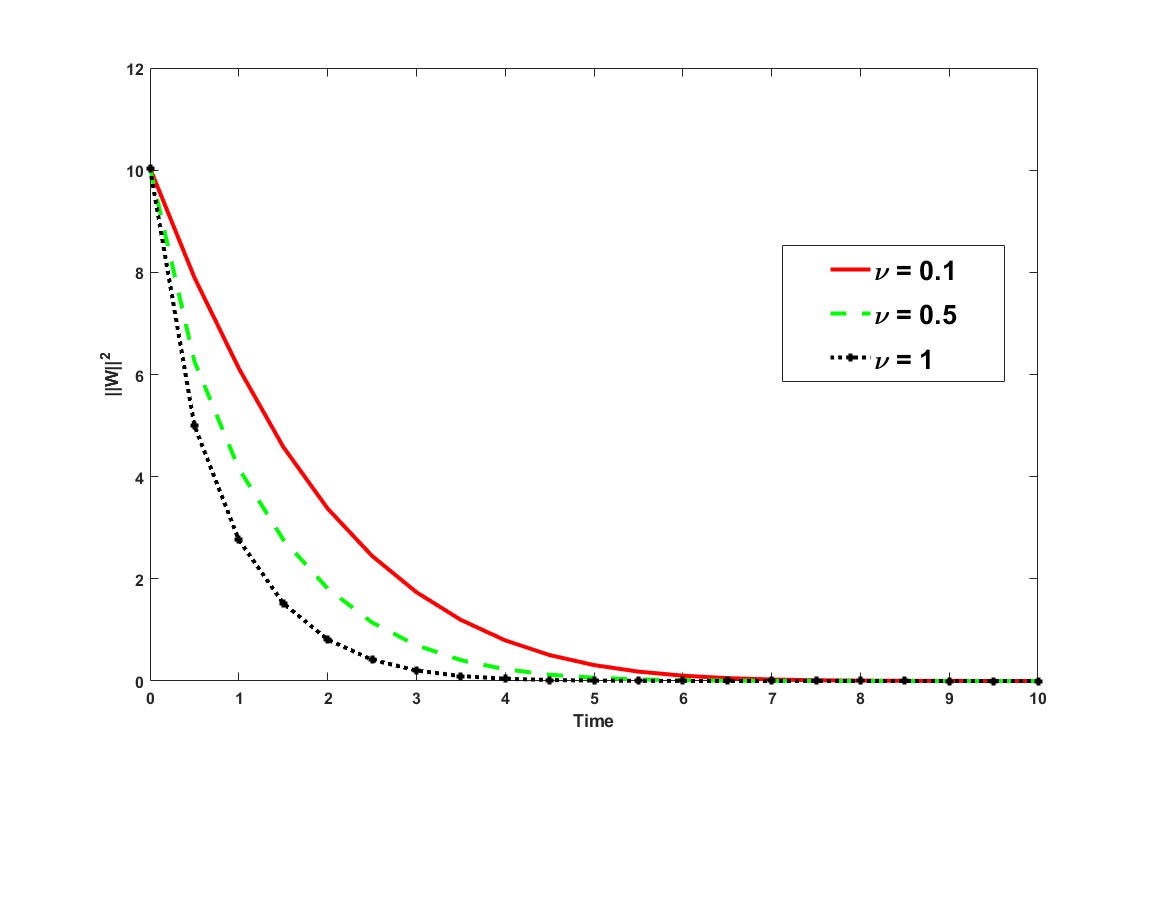}
	(ii)\includegraphics[width= 0.46\textwidth]{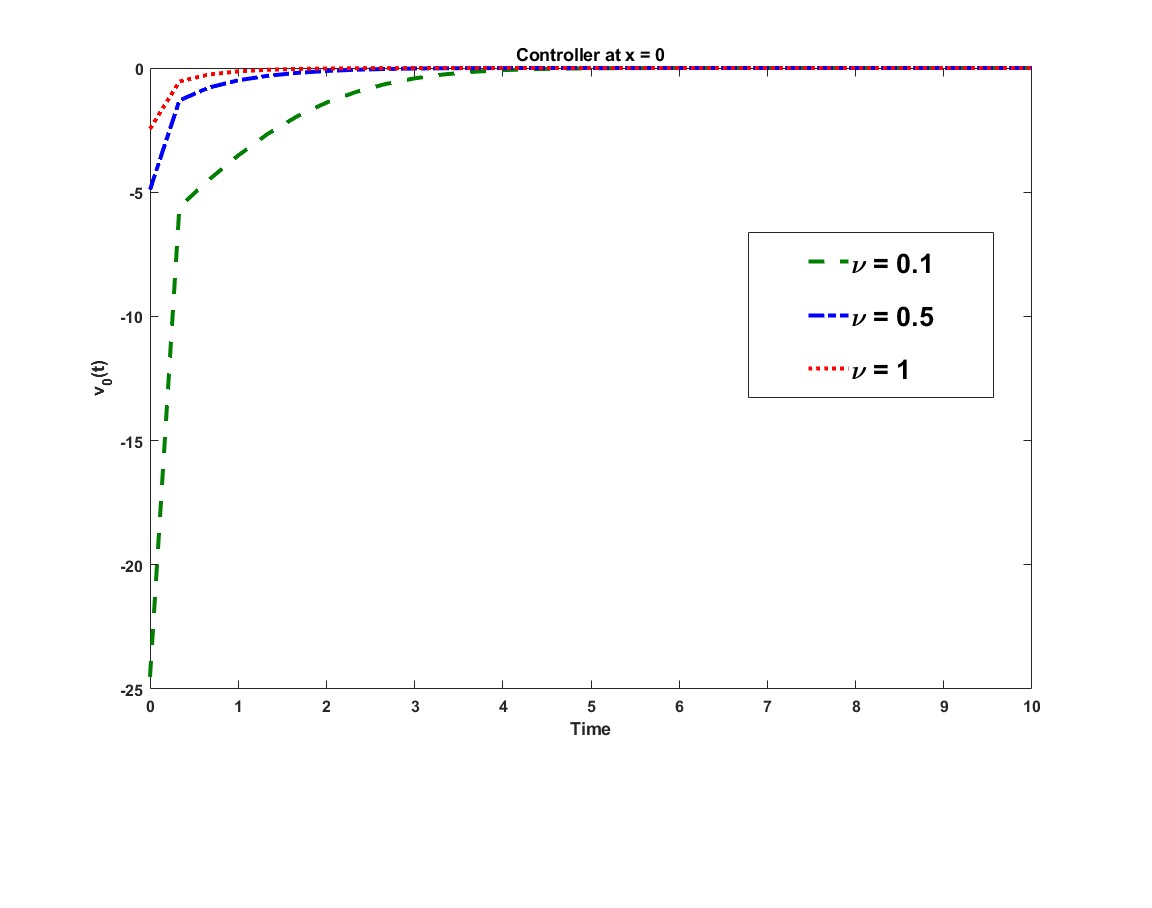}	
	(iii)\includegraphics[width= 0.50\textwidth]{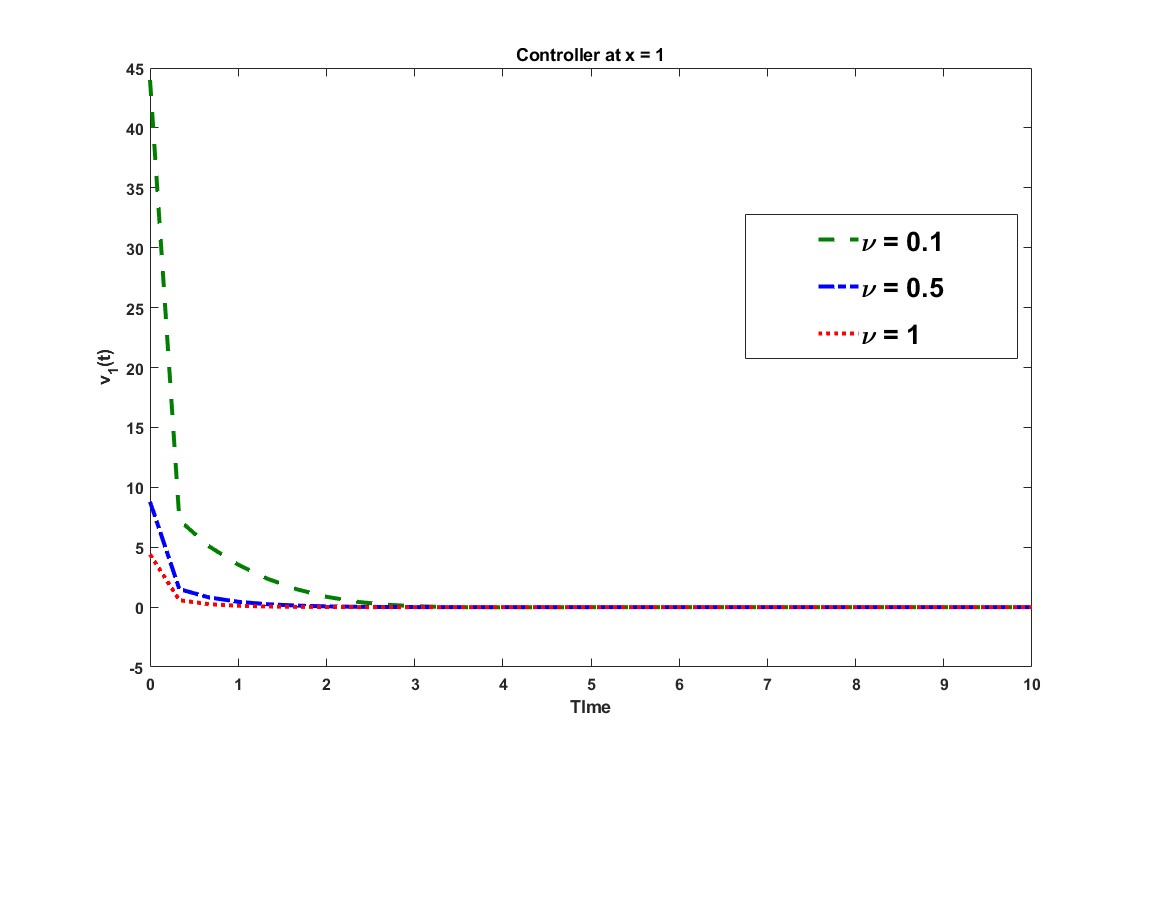}
	\caption{Example \ref{ex2}: {\bf (i)} Controlled  solution in the \(L^{2}\)-norm for various values of \(\nu\) with \(\rho=1\).  
		{\bf (ii)} Control input for different values of \(\nu\) with \(\rho=1\) at the left boundary \(x=0\).  
		{\bf (iii)} Control input for various values of \(\nu\) with \(\rho=1\) at the right boundary \(x=1\).    
	}  
	\label{fig:ex2}
\end{figure}

 Figure \ref{fig:ex3} contains the performance of the memory term coefficient keeping other parameters fixed. In Figure \ref{fig:ex3}(i), we notice that the state variable in the \(L^{2}\)-norm goes to zero for \(\rho=0,5,10,20,\) with the fixed value of the diffusion coefficient \(\nu=0.1 \).  In particular, we see that the state variable for the viscous Burgers' ($\rho=0$) equation decreases towards zero as time increases (see \cite{MR3790146} for more details). However, in the presence of a memory term,  we observe that as the values of $\rho$ increase, the state variable converges faster to the steady state solution in the $L^{2}$-norm.  In Figures \ref{fig:ex3}(ii) and (iii), we depict the controller values of the left boundary \((x=0)\) and right boundary \((x=1)\) for the numerous values of \(\rho = 0,10,20,30\) with fixed \(\nu=0.1\), respectively. Here as we increase $\rho$, the controllers have a tendency to settle at zero with higher decay in the transient phase.
\begin{figure}[h!]
	\centering
	(i)\includegraphics[width= 0.46\textwidth]{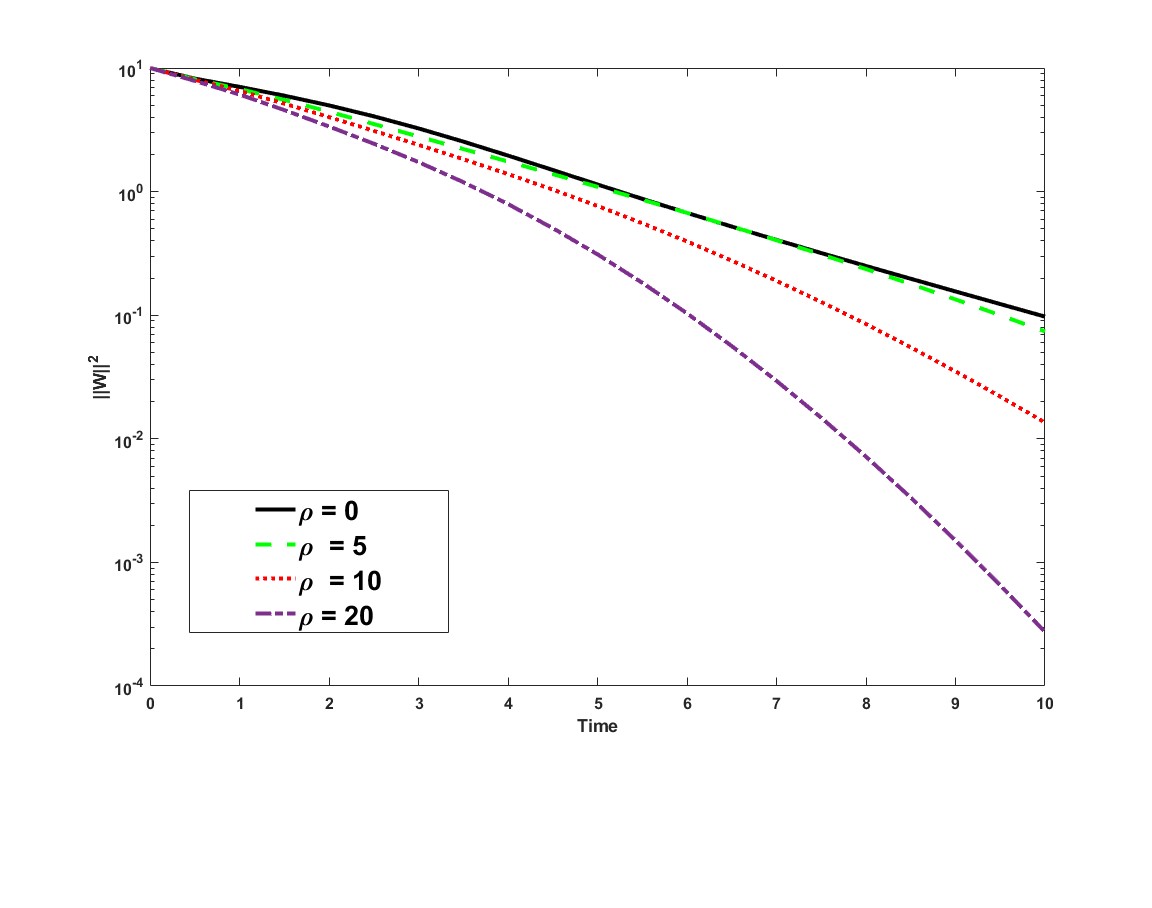}
	(ii)\includegraphics[width= 0.46\textwidth]{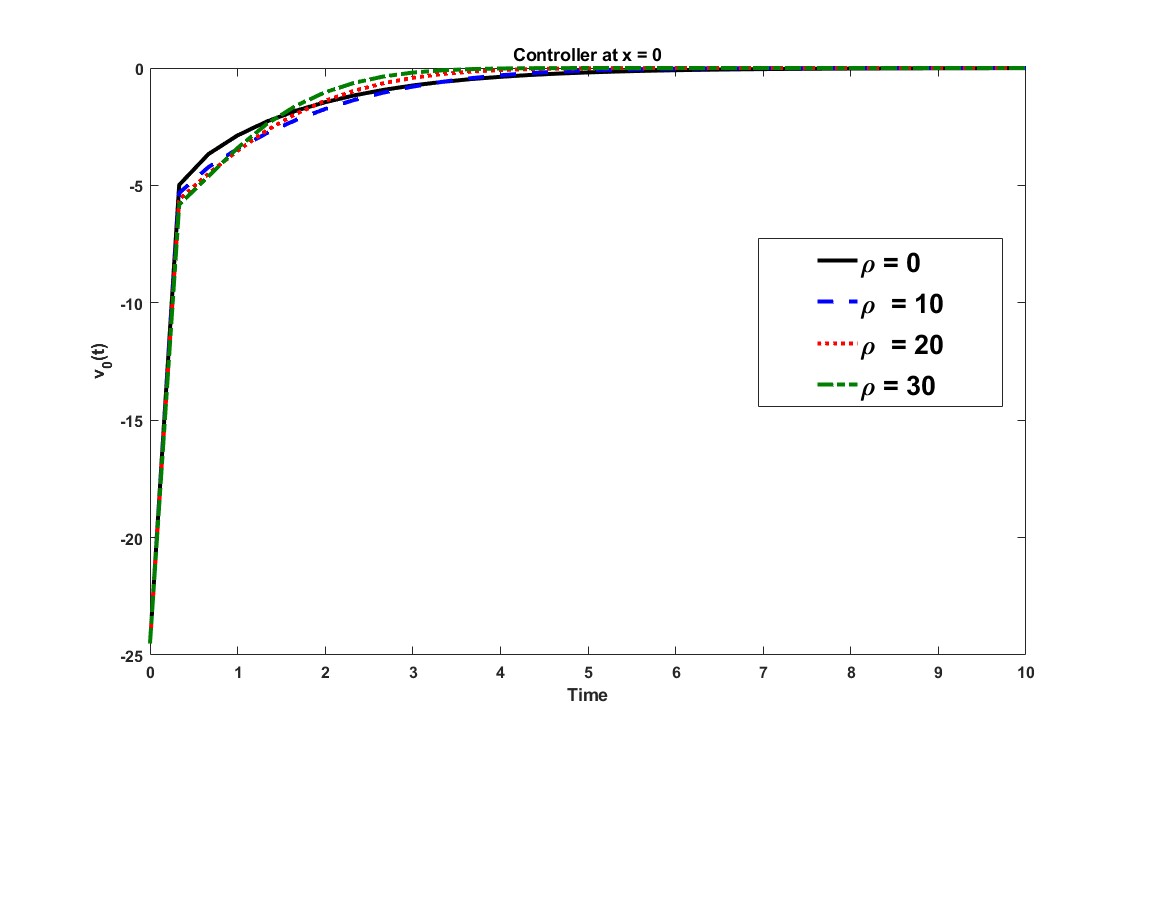}	
	(iii)\includegraphics[width= 0.50\textwidth]{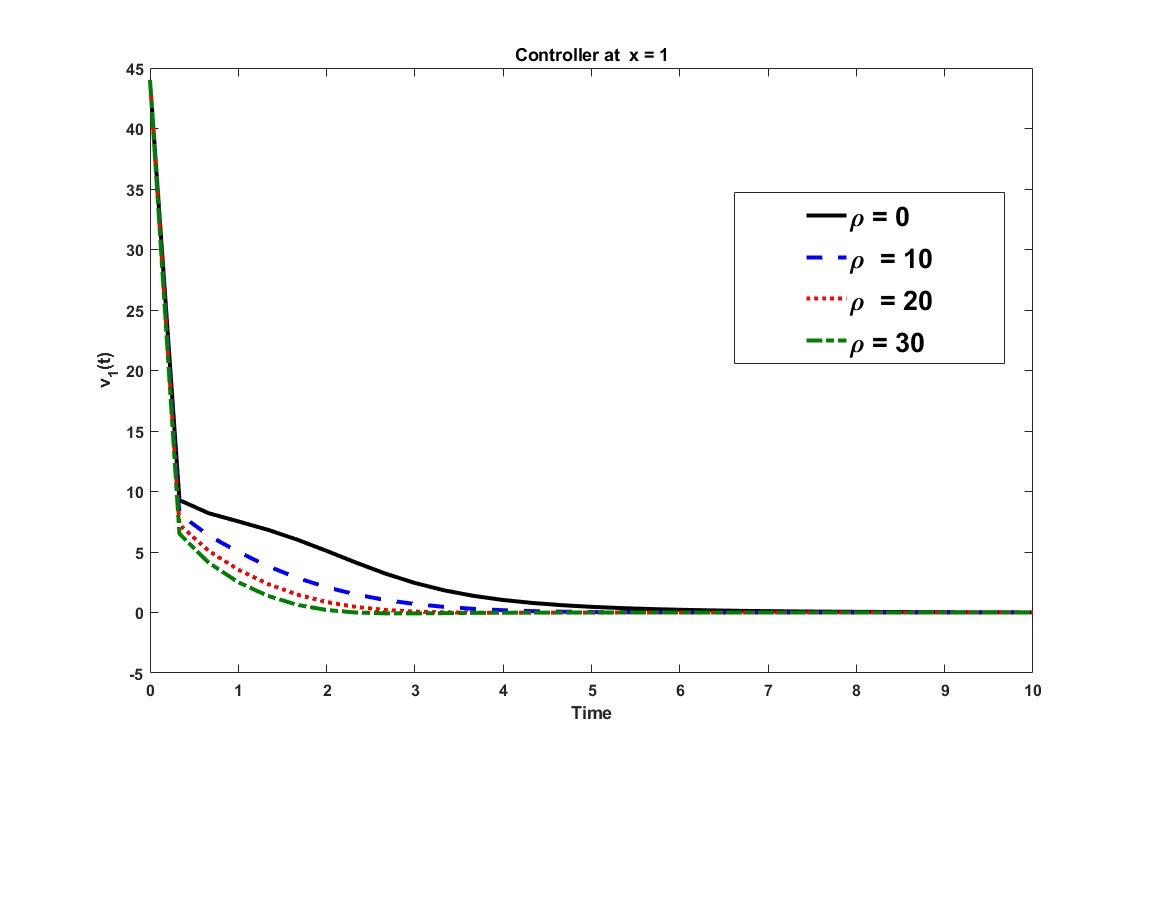}
	\caption{Example \ref{ex2}: {\bf (i)} In semi-log, controlled  solution in the \(L^{2}\)-norm for different values of  \(\rho\) with fixed \(\nu=0.1\).  
		{\bf (ii)} Control input for numerous values of \(\rho\) with fixed \(\nu=0.1\) at the left boundary \(x=0\).  
		{\bf (iii)} Control input for various values of  \(\rho\) with fixed \(\nu=0.1\) at the right boundary \(x=1\).    
	}  
	\label{fig:ex3}
\end{figure}
\section{Conclusion.}\label{6}
In this work, our aim is mainly two-fold. First, we have analyzed global stabilization of the viscous Burgers’ equation with a memory term under Neumann boundary feedback control using the control Lyapunov functional, establishing stabilization in the $L^{2}, H^{1},$ and \(H^{2}\)-norms. Next, a \(C^{0}\)-conforming finite element method has been applied for the spatial variable, keeping the time variable continuous, and an error analysis of the semi-discrete scheme has been established for the state variable and feedback control laws using the Ritz-Volterra projection. We have shown that the order of convergence for the state variable and control inputs is two. Moreover, we have examined the influence of the memory term for both the state variable and the feedback control laws.
 The work in this paper also opens the door to study other classes of nonlinear parabolic equations with memory under Neumann boundary feedback control laws.
\section*{Acknowledgments}
Sudeep Kundu gratefully acknowledges the support of the Science \& Engineering Research Board (SERB), Government of India, under the Start-up Research Grant, Project No. SRG/2022/000360.
\section*{Declarations}

\textbf{CONFlLICT OF INTEREST.} 

The authors declare no conflict of interest.

\end{document}